\documentclass[english, 12pt]{article}
\usepackage[T1]{fontenc}
\usepackage{geometry}
\usepackage[latin9]{inputenc}
\usepackage{amsthm}
\usepackage{amsmath}
\usepackage{amssymb}
\usepackage{authblk}
\usepackage[margin=2cm, labelfont = bf]{caption}
\usepackage{setspace}
\usepackage{esint}
\usepackage{enumerate}
\usepackage{etoolbox}
\usepackage{url}
\usepackage{graphicx, epstopdf}
\usepackage[english]{babel}
\usepackage{xcolor}
\usepackage{soul}
\usepackage{tikz}
\usetikzlibrary{automata,arrows,positioning,calc}

\numberwithin{equation}{section}

\makeatletter
\theoremstyle{plain}
\newtheorem{thm}{\protect\theoremname}[section]
\ifx\proof\undefined
\newenvironment{proof}[1][\protect\proofname]{\par
\normalfont\topsep6\p@\@plus6\p@\relax
\trivlist
\itemindent\parindent
\item[\hskip\labelsep
\scshape
#1]\ignorespaces
}{%
\endtrivlist\@endpefalse
}
\providecommand{\proofname}{Proof}
\fi
\theoremstyle{plain}
\newtheorem{lem}[thm]{\protect\lemmaname}
\theoremstyle{plain}
\newtheorem{prop}[thm]{\protect\propositionname}
\theoremstyle{plain}
\newtheorem{conjecture}[thm]{\protect\conjecturename}
\theoremstyle{definition}
\newtheorem{defn}[thm]{\protect\definitionname}
\theoremstyle{remark}
\newtheorem{rem}[thm]{\protect\remarkname}
\theoremstyle{plain}
\newtheorem{cor}[thm]{\protect\corollaryname}
\theoremstyle{definition}
\newtheorem{example}[thm]{\protect\examplename}
\numberwithin{figure}{section}

\makeatother

\usepackage{babel}
\providecommand{\conjecturename}{Conjecture}
\providecommand{\corollaryname}{Corollary}
\providecommand{\definitionname}{Definition}
\providecommand{\examplename}{Example}
\providecommand{\lemmaname}{Lemma}
\providecommand{\propositionname}{Proposition}
\providecommand{\remarkname}{Remark}
\providecommand{\theoremname}{Theorem}

\begin{document}


\title{A McKean--Vlasov equation with positive feedback and blow-ups}

\author[1]{Ben Hambly}
\author[2]{Sean Ledger}
\author[1]{Andreas S{\o}jmark}
\affil[1]{Mathematical Institute, University of Oxford, Oxford, OX2 6GG, UK}
\affil[2]{School of Mathematics, University of Bristol and Heilbronn Institute for Mathematical Research, Bristol, BS8 1TW, UK }

\date{23\textsuperscript{rd} January 2018}

\maketitle

\begin{abstract}
We study a McKean--Vlasov equation arising from a mean-field model of a particle system with positive feedback. As particles hit a barrier they cause the other particles to jump in the direction of the barrier and this feedback mechanism leads to the possibility that the system can exhibit contagious blow-ups. Using a fixed-point argument we construct a differentiable solution up to a first explosion time. Our main contribution is a proof of uniqueness in the class of c\`{a}dl\`{a}g functions, which confirms the validity of related propagation-of-chaos results in the literature. We extend the allowed initial conditions to include densities with any power law decay at the boundary, and connect the exponent of decay with the growth exponent of the solution in small time in a  precise way. This takes us asymptotically close to the control on initial conditions required for a global solution theory. A novel minimality result and trapping technique are introduced to prove uniqueness. 
\end{abstract}

\section{Introduction} 
\label{Sect_Intro}

This paper concerns a McKean--Vlasov problem, formulated probabilistically as  
\begin{align}
\label{eq:Intro_MVproblem}
\begin{cases}
 X_t = X_0 + B_t - \alpha L_t \\
\tau = \inf\{ t \geq 0 : X_t \leq 0 \}  \\
L_t = \mathbb{P}(\tau \leq t), 
\end{cases}
\end{align}
where $\alpha \in \mathbb{R}$ is a constant, $B$ is a standard Brownian motion and $X_0$ is an independent random variable distributed on the positive half-line. We denote the law of $X_0$ by $\nu_0$.
A~solution to this  problem is a deterministic and initially zero c\`adl\`ag function $t \mapsto L_t$ that is increasing and for which (\ref{eq:Intro_MVproblem}) holds for any Brownian motion $B$. Viewing~(\ref{eq:Intro_MVproblem}) as an SDE~in $X$, notice that there is no distinction to be made between strong and weak notions of solution: knowing $L$ fixes the law of $X$ and, together with any Brownian motion, $L$ fixes a pathwise construction of $X$. When $\alpha > 0$, the equations have a positive feedback effect and this is the case we consider here --- the situation for $\alpha \leq 0$ is classical and existence and uniqueness of smooth solutions is known \cite{carrillo2011, carrillo2013, carrillo2015}.


Our motivation for studying (\ref{eq:Intro_MVproblem}) comes from mathematical finance, where it can be used as a simple model for contagion in large financial networks or large portfolios of defaultable entities. To illustrate how  (\ref{eq:Intro_MVproblem}) may emerge
in this context, consider a large system of $N$ banks. Following the structural approach to credit risk, we say that the
$i$'th bank defaults when its asset value, $A_{t}^{i}$, hits a default
barrier, $D_{t}^{i}$. This gives rise to the notion of \emph{distance-to-default}, for which a simple model could be of the form
\[
Y_{t}^{i}:=\log(A_{t}^{i})-\log(D_{t}^{i})=X_{0}^{i}+{\textstyle \int_{0}^{t}b(s)ds}+{\textstyle \int_{0}^{t} \sigma(s)dB_{s}^{i}},\qquad i=1,\ldots,N,
\]
where $X_0^{1},\ldots,X_0^{N}$ are i.i.d.~copies of $X_0$ and $B^{1},\ldots,B^{N}$ are independent Brownian motions. Next, we can introduce an element of contagion with a model in which the default of one bank causes the other banks to lose a proportion $\alpha/N$ of their assets. For large $N$, we have
$1-\alpha/N\simeq\exp\left\{ -\alpha/N\right\}$, so the new asset values, $\hat{A}$, are then defined by
\[
\hat{A}_{t}^{i,N}:=\prod_{j=1}^{N}\exp\Bigl\{-\frac{\alpha}{N}\mathbf{1}_{t\geq\tau^{j,N}}\Bigr\} A_{t}^{i}=\exp\Bigl\{-\frac{\alpha}{N}\sum_{j=1}^{N}\mathbf{1}_{t\geq\tau^{j,N}}\Bigr\} A_{t}^{i},\qquad t<\tau^{i,N},
\]
where $\tau^{i,N}:= \{ t>0:X_{t}^{i,N}\leq0 \}$ for $X_{t}^{i,N}:=\log(\hat{A}_{t}^{i,N})-\log(D_{t}^{i})$.
After taking logarithms, it follows that the new distances-to-default, $X^{i,N}$, satisfy
\begin{equation}
\label{eq:Intro_FiniteSystem_Def}
dX_{t}^{i,N} = b(t)dt +  \sigma(t)dB_{t}^{i}-\alpha dL_{t}^{N}\quad\text{with}\quad L_{t}^{N}=\frac{1}{N}\sum_{j=1}^{N}\mathbf{1}_{t\geq\tau^{j,N}},
\end{equation}
for $i=1,\dots,N$. If we let $N\rightarrow\infty$, then the same arguments as in \cite{dirt_SPA_2015}
show that we can recover solutions to  (\ref{eq:Intro_MVproblem}) --- with the corresponding drift and volatility --- as topological limit points of the particle system (\ref{eq:Intro_FiniteSystem_Def}), and these solutions are global: they exist for all $t \geq 0$. As regards the form of (\ref{eq:Intro_FiniteSystem_Def}), we will concentrate the analysis in this paper on the simplest case (\ref{eq:Intro_MVproblem}) and then we devote the final Section \ref{Sect_GeneralCoeff} to a discussion of more general coefficients.

 The first version of the problem (\ref{eq:Intro_MVproblem}) appeared in the mathematical neuroscience literature as a mean-field limit of a large network of electrically coupled neurons 
\cite{carrillo2011, carrillo2013, dirt_annalsAP_2015, dirt_SPA_2015}. In this setting, each neuron is identified with an electrical potential (given by an SDE) and when it reaches a threshold voltage, the neuron fires an electrical signal to the other neurons, which then become excited to higher voltage levels. After reaching the threshold, the neuron is instantaneously reset to a predetermined value and it then continues to evolve according to this rule indefinitely. Therefore, the model is different to our setting as we do not reset the mean-field particle in (\ref{eq:Intro_MVproblem}), however, the essential mathematical difficulties from the positive feedback remain common to both models. 
With regard to the financial framework introduced above, we note that a similar model for default contagion (with constant drift and volatility) was recently proposed in \cite{nadtochiy_shkolnikov_2017}.


Mathematically, the McKean--Vlasov problem (\ref{eq:Intro_MVproblem}) can be recast in a number of ways. We denote the law of $X_t$ killed at the origin 
by $\nu_t$ and set $\nu_t(\phi) := \mathbb{E}[\phi(X_t)\mathbf{1}_{t < \tau}]$, for suitable test functions $\phi$. Then a simple application of It\^{o}'s formula  gives the nonlinear PDE
\begin{align}
\label{eq:Intro_PDEproblem}
\begin{cases}
\nu_t(\phi) = \nu_0(\phi) + \tfrac{1}{2}\int_0^t \nu_s(\phi'')ds - \alpha \int_0^t \nu_s(\phi')dL_s \\
L_t = 1 - \int_0^\infty \nu_s(dx),
\end{cases}
\end{align}
for $\phi \in C^2$ with $\phi(0)=0$ and $t \geq 0$. Writing $V_t$ for the density of $\nu_t$ (which exists by Proposition \ref{MinimalJumps_Prop_DensityExists}), we can formally integrate by parts in (\ref{eq:Intro_PDEproblem}) to find the Dirichlet problem
\begin{align}
\label{eq:Intro_DensityPDE}
\partial_{t}V_{t}(x)={\textstyle \frac{1}{2}}\partial_{xx}V_{t}(x)+\alpha L_{t}^{\prime}\partial_{x}V_{t}(x),\qquad L_{t}^{\prime}={\textstyle \frac{1}{2}}\partial_{x}V_{t}(0),\qquad V_{t}(0)=0.
\end{align}
 In other words, the law of $X$ solves the heat equation with a drift term proportional to the flux  across the Dirichlet boundary at zero --- the latter being a highly singular nonlocal nonlinearity. Setting $v(t,x):=-V_t(x-\alpha L_t)$ in (\ref{eq:Intro_DensityPDE}), the equations for $v$ and $L$ can be viewed as a Stefan problem with supercooling on the semi-infinite strip $(\alpha L_t,\infty)$. From this point of view, it is known that  $L'_t$ may explode in finite time as has been analysed (on a finite strip) in the series of papers \cite{fasano_primicerio, howi_ock1, howi_ock2, herrero_velazquez}.

 A third characterisation of (\ref{eq:Intro_MVproblem}) is as the solution to the integral equation
\begin{align}
\label{eq:Intro_IEproblem}
\int_0^\infty \Phi\Big( -\frac{x - \alpha L_t}{t^{1/2}} \Big) \nu_0(dx)
	= \int^t_0 \Phi\Big( \alpha \frac{L_t-L_s}{(t-s)^{1/2}} \Big) dL_s,
\end{align}
where $\Phi$ is the  c.d.f.~of the Normal distribution. This is a Volterra integral equation of the first kind, and a derivation can be obtained by following \cite{peskir_2002}. The formulation that is most helpful here is to view the problem as a fixed point of the map $\Gamma$ defined by
\begin{align}
\label{eq:Intro_DefnOfGamma}
\begin{cases}
 X_t^\ell = X_0 + B_t - \alpha \ell_t \\
\tau^\ell = \inf\{ t \geq 0 : X_t^\ell \leq 0 \}  \\
\Gamma[\ell]_t = \mathbb{P}(\tau^\ell \leq t), 
\end{cases}
\end{align}
that is, a solution to $\Gamma[L] = L$. We will use this map in stating and proving our main theorems. The key is to find a suitable space on which $\Gamma$ stabilises and to show that it is contractive (Theorems \ref{Intro_Thm_Minimality} \& \ref{Intro_Thm_FixedPoint}). Note that \cite{dirt_annalsAP_2015, nadtochiy_shkolnikov_2017} also take this approach to the problem, but our techniques differ in that they are entirely probabilistic and do not rely on PDE estimates.

\begin{figure}
\begin{center}
\includegraphics[width=0.65\textwidth]{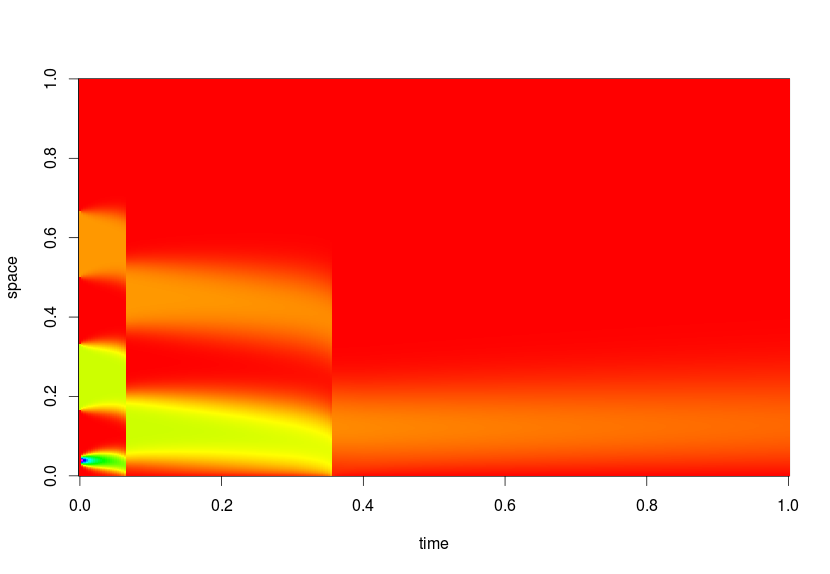}
\caption{\label{fig:Intro_BlowUpExample} Example of a solution to (\ref{eq:Intro_PDEproblem}, \ref{eq:Intro_DensityPDE}) showing two blow-up times. Pixel intensity represents the value of the solution density at that space-time coordinate. The initial condition is a linear combination of indicator functions of three disjoint sets. }
\end{center} \vspace{-10pt}
\end{figure}

A very interesting feature of the problem (\ref{eq:Intro_MVproblem}) is that it exhibits a phase transition in the continuity of solutions as the feedback strength, $\alpha$, increases. The methods of \cite{dirt_annalsAP_2015}, for the neuroscience version of the problem, show that for $\alpha$ sufficiently small (and $\nu_0 = \delta_x$ with $x > 0$) there is a unique solution to (\ref{eq:Intro_MVproblem}) in the class of  continuously differentiable functions. 
On the other hand, the extremely simple proof below (modified from \cite{carrillo2011}) shows that (\ref{eq:Intro_MVproblem}) cannot have a continuous solution for values of $\alpha$ that are sufficiently large. When this is the case, the positive feedback becomes too great and at some point in time the loss process, $L$, undergoes a jump discontinuity, which we call a \emph{blow-up}. In other words, in an infinitesimal period of time, a macroscopic proportion of the mass in the PDE (\ref{eq:Intro_PDEproblem}) is lost at the boundary --- i.e.~the system blows up (see Figure \ref{fig:Intro_BlowUpExample}). It is intriguing that this result can be proved so simply and with no technical estimates, although of course the threshold obtained below is not sharp and the proof does not reveal anything about the nature of the blow-ups.

\begin{thm}[Blow-up for large $\alpha$]
\label{Intro_Thm_BlowUP}
Let $m_0 := \int_0^\infty x\nu_0(dx)$. If $\alpha > 2m_0$, then any solution to (\ref{eq:Intro_MVproblem}) cannot be continuous for all times.
\end{thm}

\begin{proof}
For a contradiction, suppose $L$ solves (\ref{eq:Intro_MVproblem}) and is continuous. Stopping $X$ at $\tau$, we then get
\[
0 \leq X_{t \wedge \tau} = X_0 + B_{t \wedge \tau} - \alpha L_{t \wedge \tau}.
\]
Taking expectations and rearranging,
\[
m_0 \geq \alpha \mathbb{E}[L_{t \wedge \tau}]
	\to \alpha \mathbb{E}[L_{\tau}]
	= \alpha\int^\infty_0 L_s dL_s\qquad \textrm{as } t \to \infty, 
\]
where we note that $\tau < \infty$ a.s.~and $L_\infty = 1$, since Brownian motion hits every level with probability 1. As $L$ is continuous and increasing, the integral can be computed exactly, so we get
\[
m_0 \geq \tfrac{1}{2}\alpha (L_\infty^2 - L_0^2) = \tfrac{1}{2}\alpha,
\]
which is the required contradiction.
\end{proof}

From Theorem \ref{Intro_Thm_BlowUP} it is clear that, in general, we cannot restrict our search to solutions of (\ref{eq:Intro_MVproblem}) that are continuous and so we must allow c\`{a}dl\`{a}g solutions. If a solution has a jump of size $\Delta L_t$ at time $t$,  then the instantaneous loss must equal the mass of $\nu_{t-}$ absorbed at the boundary after a translation by  $-\alpha\Delta L_t$, that is:
\begin{equation}
\label{eq:Intro_JumpCondition}
\nu_{t-}(0,\alpha \Delta L_{t} ) = \Delta L_t,
\end{equation}
see Figure \ref{fig:Intro_Typical_IC}. Unfortunately, this equation alone is not sufficient to determine the jump sizes of a discontinuous solution. In particular, $\Delta L_{t} = 0$ is always a solution, but in general this is an invalid jump size by Theorem \ref{Intro_Thm_BlowUP}. To continue the solution after a blow-up, we must decide how to choose the jump size from the solution set of (\ref{eq:Intro_JumpCondition}). In \cite{dirt_SPA_2015} the authors introduce the term \emph{physical solution} for a solution, $L$, that satisfies 
\begin{equation}
\label{eq:Intro_PhysicalJumpCondition}
\Delta L_t = \inf\{ x \geq 0 : \nu_{t-}(0,\alpha  x) < x \}
\end{equation}
for all $t \geq 0$. The next result justifies that this is a natural condition, as it is the smallest possible choice of jump size that admits c\`adl\`ag solutions (see Section \ref{Sect_MinimalJmmps} for a proof and further discussion).

\begin{figure}
	\begin{center}
		\begin{tikzpicture}[scale = 0.9]
			\draw[->] (-0.5,0) -- (4,0) node [right] {$x$};
			\draw[->] (0,0) -- (0,2.5) node [above] {$V_{t-}(x)$};
			\draw[-,  line width = 0.2mm] plot [smooth] coordinates { (0,0) (0.05, 0.7) (0.36, 0.17) (0.95, 0.14)  (1.3, 0.5) (1.8,2) (2.4,0.4) (4, 0.1) } ;
			
			\draw[->] (-0.5+6,0) -- (4+6,0) node [right] {$x$};
			\draw[->] (0+6,0) -- (0+6,2.5) node [above] {$V_{t-}(x + \alpha \Delta L_t)$};
			\draw[-,  line width = 0.2mm] plot [smooth] coordinates {   (0+6-0.36,0) (0.05+6-0.36, 0.7) (0.36+6-0.36, 0.17) (0.95+6-0.36, 0.14)  (1.3+6-0.36, 0.5) (1.8+6-0.36,2) (2.4+6-0.36,0.4) (4+6-0.36, 0.1) } ;
			\draw[-, fill = gray] plot coordinates { (0.01+6-0.36,0.01) (0.024+6-0.36, 0.52) (0.036+6-0.36, 0.62)  (0.053+6-0.36, 0.68) (0.07+6-0.36, 0.68) (0.1+6-0.36, 0.62) (0.22+6-0.36, 0.36) (0.29+6-0.36, 0.23) (0.36+6-0.36, 0.16) (6,0.01) } ;
			\draw[<-, dashed] (0+6-0.5,0.5) -- (0+6-0.5-1.5,1.1);
			\draw[] (0+6-0.5,1.0) node [above left] {$\nu_{t-}(0,\alpha \Delta L_t)$};
			
			\draw[->] (-0.5+12,0) -- (4+12,0) node [right] {$x$};
			\draw[->] (0+12,0) -- (0+12,2.5) node [above] {$V_{t}(x)$};
			\draw[-,  line width = 0.2mm] plot [smooth] coordinates { (0.36+12-0.36, 0.17) (0.522+12-0.36, 0.1136)
				(0.95+12-0.36, 0.14)  (1.3+12-0.36, 0.5) (1.8+12-0.36,2) (2.4+12-0.36,0.4) (4+12-0.36, 0.1) } ;
		\end{tikzpicture}
		\caption{\label{fig:Intro_Typical_IC} On the left, $V_{t-}$ is the density just before a jump of size $\Delta L_t$. This density is then translated by $\alpha \Delta L_t $ and the mass falling below the boundary at zero equals the change in the loss, which gives (\ref{eq:Intro_JumpCondition}). After the jump, the system is restarted from the density on the right. Notice that, in general, this new initial condition will not vanish at the origin. }
	\end{center} \vspace{-10pt}
\end{figure}
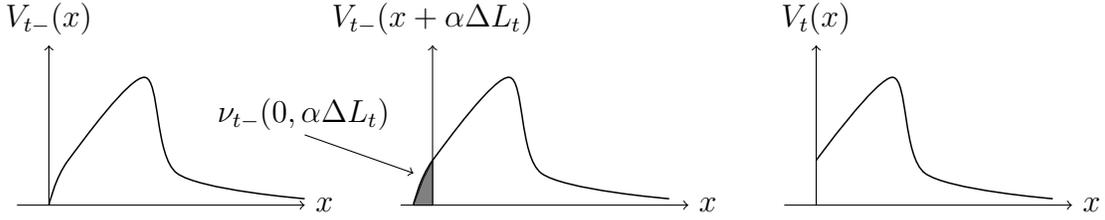

\begin{prop}[Physical solutions have minimal jumps]
\label{Intro_Prop_MinimalJumps}
Suppose $L$ is any c\`adl\`ag process satisfying (\ref{eq:Intro_MVproblem}). Then 
\[
\Delta L_t \geq \inf\{ x \geq 0 : \nu_{t-}(0,\alpha  x) < x \},
\]
for every $t \geq 0$. In particular, if the right-hand side is non-zero for some $t \geq 0$, then $L$ has a blow-up at time $t$.
\end{prop}

\noindent It is helpful to consider the density function, $V_{t-}$, of $\nu_{t-}$ in light of Proposition \ref{Intro_Prop_MinimalJumps}. If, at some time $t$, $V_{t-}$ is greater than or equal to the critical value of $\alpha^{-1}$ on a non-zero interval about the origin, then a blow-up in $L$ is forced to occur at that time.

In \cite{dirt_SPA_2015} it is shown that there exist (global) physical solutions to the neuroscience version of equation (\ref{eq:Intro_MVproblem}), for any $\alpha > 0$, albeit with an initial measure $\nu_0$ that vanishes in a neighbourhood of zero. Those solutions arise as topological limit points of a corresponding finite particle system analogous to (\ref{eq:Intro_FiniteSystem_Def}).  Consequently, the authors are unable to establish regularity results on the solutions. The main advantage of a fixed point argument is that solutions are guaranteed to have known regularity, at least on a small time interval.

Our motivation in this paper is to make progress towards the following conjecture:

\begin{conjecture}[Global uniqueness]
\label{Into_Conj_GlobalUniqueness}
Solutions to (\ref{Intro_Thm_BlowUP}) that satisfy the physical jump condition (\ref{eq:Intro_PhysicalJumpCondition}) are unique. Furthermore, between jump times the solution is continuously differentiable, and we predict $\sqrt{t}$-singularities immediately before and after jumps, which is to say: For all $t_0 > 0$ that satisfy $\Delta L_{t_0} > 0$, we have $L'_{t_0 + h} = O(|h|^{-1/2})$, as $h \to 0$.
\end{conjecture} 

\noindent Currently, we are far from proving Conjecture \ref{Into_Conj_GlobalUniqueness}, and the major obstruction concerns the initial conditions. To see why, notice that after a jump has taken place, we must, in general, restart the system from an initial law that has a density which does not vanish at the origin (see Figure \ref{fig:Intro_Typical_IC}, as well as Proposition \ref{MinimalJumps_Prop_DensityExists} concerning the existence of a density for $\nu$ at all times). In fact, without further analysis, all we can say about the measure $\nu_t$ after a blow-up at time $t$ is that
\[
\inf\{x \geq 0 : \nu_t(0,\alpha x) < x\} = 0.
\]
Therefore, to attack Conjecture \ref{Into_Conj_GlobalUniqueness} it is necessary to make progress towards the following simpler goal:

\begin{conjecture}[Uniqueness for non-vanishing initial laws]
\label{Into_Conj_UniquenessNonVanish}
Suppose $\nu_0$ has a density and satisfies $\inf\{x \geq 0 : \nu_0(0,\alpha x) < x\} = 0$. Then the solution, L, to (\ref{eq:Intro_MVproblem}) is unique and $C^1$ up to a small time, and it satisfies  $ L_t' = O(t^{-1/2})$ as $t \downarrow 0$. 
\end{conjecture} 

Initial conditions that do not vanish at the origin are currently outside the scope of known results, and the results that do exist only tackle the uniqueness in a too restrictive class of candidate solutions. In the literature, the closest to Conjecture \ref{Into_Conj_UniquenessNonVanish} is \cite[Thm.~2.6]{nadtochiy_shkolnikov_2017} which is established for a slight variant of our problem: Given an initial density $V_0 \in H^1(0,\infty)$ with $V_0(0)=0$, it gives existence and uniqueness up to the first time the $L^2$-norm of $L^{\prime}$ explodes in the class of candidate solutions for which $L\in H^{1}(0,t_{0})$ for some $t_0>0$. Here $H^1(a,b)$ denotes the usual Sobolev space with one weak derivative in $L^2(a,b)$. We note that these conditions on $V_0$ imply that it decays like $o(x^{1/2})$ near the origin, so it is far from what is needed in order to proceed to a global uniqueness theory.  The aforementioned was preceded by \cite[Thm.~4.1]{dirt_annalsAP_2015}, which gives existence and uniqueness on a small time interval in the class of $C^1$ solutions, starting from an initial density that decays like $O(x)$ as $x \to 0$. The connection between the boundary decay of the initial condition and the short-time regularity of solutions will feature prominently in our results below.

Moreover, the existing results in the literature leave open the question of whether there is (short-time) uniqueness in the wider class of c\`{a}dl\`{a}g solutions. This is more than just a technical curiosity: Indeed, the physical and financial motivations for (\ref{eq:Intro_MVproblem}) derive from the corresponding particle system, as presented in (\ref{eq:Intro_FiniteSystem_Def}), but the requisite regularity of its limit points is not known and could be difficult to verify a priori --- all we know is that they are c\`{a}dl\`{a}g. Thus, it obstructs the full convergence in law of the finite system (i.e.~the propagation of chaos).

\subsection{Main results}

Our contribution here is to construct an $H^1$ solution to (\ref{eq:Intro_MVproblem}) up to the first time its $H^1$ norm explodes, and to show that it is the unique solution on this time interval amongst all possible c\`{a}dl\`{a}g solutions (Theorem \ref{Intro_Thm_MainUniqueness}). We allow initial densities with any power law decay at the origin, that is, decay of order $O(x^\beta)$ as $x\rightarrow0$ for some $0 < \beta < 1$. To be precise, we consider initial conditions, $\nu_0$, that have a bounded initial density, $V_0$, for which we can find constants $C, D, x_\star >0$ such that
\begin{equation}
\label{eq:Intro_DensityAssumption}
V_0(x) \leq Cx^\beta \mathbf{1}_{x \leq x_\star} + D \mathbf{1}_{x > x_\star},
	\qquad \textrm{for every } x > 0. 
\end{equation}
A key observation is that it is only the values of $C$, $x_\star$ and $\beta$ that determine the behaviour of solutions in short time (Theorem \ref{Intro_Thm_FixedPoint}).
The reason for assuming an initial density is that this state is reached after any non-zero period of time anyway (Proposition \ref{MinimalJumps_Prop_DensityExists}). Below, we state our main results chronologically to make it clear how they are connected.

First we introduce the following subsets of $H^1$: 

\begin{defn}
\label{Intro_Def_SetS}
For $\gamma \in (0,\tfrac{1}{2})$, $A > 0$ and $t_0 >0 $, let $\mathcal{S}(\gamma , A, t_0)$ denote the subset of $H^1(0,t_0)$ given by
\[
\mathcal{S}(\gamma , A, t_0)
	:= \{ \ell \in H^1(0,t_0) \textrm{ such that } \ell'_t \leq A t^{-\gamma} \textrm{ for almost all } t \in [0,t_0] \}.
\] 
\end{defn}

\noindent Our first argument is to show that, on these sets, $\Gamma$ is an $L^\infty$-contraction. The proof follows by comparing the first hitting times of a single Brownian motion driven by two different drift functions (Proposition \ref{Unique_Prop_Comparison}), which is a coupling of the $X^\ell$ processes in (\ref{eq:Intro_DefnOfGamma}) to the same Brownian motion $B$. As a by-product of this method we can deduce that differentiable solutions are minimal in the class of potential c\`{a}dl\`{a}g solutions, which is an important ingredient in the main uniqueness result (Theorem \ref{Intro_Thm_MainUniqueness}). The power of the technique lies in the ability to estimate the positive part of the difference $(L_t - \bar{L}_t)^+$, for two solutions $L$ and $\bar{L}$, assuming only the regularity properties of $L$ and not $\bar{L}$.

\begin{thm}[Contraction and minimality]
\label{Intro_Thm_Minimality}
For any $\gamma \in (0,\tfrac{1}{2})$ and $A>0$, there exists $t_0 > 0$ such that for all $\ell, \bar{\ell} \in \mathcal{S}(\gamma, A,t_0)$
\[
\Vert \Gamma[\ell] - \Gamma[\bar{\ell}] \Vert_{L^\infty(0,t_0)} \leq \tfrac{1}{2} \Vert \ell - \bar{\ell} \Vert_{L^\infty(0,t_0)}.
\]

Moreover, if there exists a solution $L \in \mathcal{S}(\gamma,A,t_0)$ to (\ref{eq:Intro_MVproblem}) for some $\gamma\in(0,\tfrac{1}{2})$, $A>0$ and $t_0 > 0$ and another solution $\bar{L}$ that is c\`{a}dl\`{a}g, then $\bar{L}_t \geq L_t$ for all $t < t_0$. 
\end{thm}

To construct a fixed point solution, we find choices of the parameters $\gamma$, $A$ and $t_0$ such that the map $\Gamma$ stabilises on the set $\mathcal{S}(\gamma, A, t_0)$. An interesting product of our technique is that we are able to recover the exact regularity of the solutions at time zero from the decay of the initial condition near the Dirichlet boundary. Thus, the main factor in the short-time growth of solutions is the behaviour of the heat equation with absorbing boundary conditions and not the feedback effect in the model. This should be an indication that something new is required to tackle Conjecture \ref{Into_Conj_GlobalUniqueness}. Crucially, we rely on Girsanov's Theorem to control $t \mapsto B_t - \alpha \ell_t$, for which we require $\ell \in H^1(0,t_0)$. Notice, however, that if $V_0(+0) > 0$ as in the hypothesis of Conjecture \ref{Into_Conj_GlobalUniqueness}, then we must have $\ell_t \geq \mathrm{const.} \times \sqrt{t}$ (by comparison with the case $\alpha = 0$), and so we can no longer expect this approach to be viable.

\begin{thm}[Stability and fixed point]
\label{Intro_Thm_FixedPoint}
There exists a constant $K > 0$ depending only on $\beta$, $C$ and $x_\star$ such that for all $\varepsilon > 0$ there exists $t_0$ for which
\[
\Gamma : \mathcal{S}(\tfrac{1-\beta}{2}, K + \varepsilon, t_0) \to \mathcal{S}(\tfrac{1-\beta}{2}, K + \varepsilon, t_0).
\]

Hence $\Gamma$ has a fixed point, $L = \Gamma[L]$, in $\mathcal{S}(\tfrac{1-\beta}{2}, K + \varepsilon, t_0)$ and this solution is unique in the class of candidate solutions in
$\bigcup_{0<\gamma<1/2, \thinspace A >0} \mathcal{S}(\gamma, A, t_0)$.  
\end{thm}

This short-time fixed point construction can be extended by observing that if $L$ is in the space $H^1(0,t_0)$, then, by appealing to Girsanov's Theorem, we can show that the density $V_{t_0-}$ must have decay of order $O(x^\beta)$ near the Dirichlet boundary. In other words, we can recover the regularity of the initial condition at time $t_0$, and so exactly the same argument can be applied to the system restarted from time $t_0$. Consequently, we can extend the solution constructed in Theorem \ref{Intro_Thm_FixedPoint} onto another small non-zero time interval and iterate this procedure, so long as the solution we obtain is in $H^1$. The catch that prevents this procedure from giving a differentiable solution for all times (recall that this would contradict Theorem \ref{Intro_Thm_BlowUP}) is that we lose control of the constants in the relevant $\mathcal{S}$ subsets. Hence the construction only applies up to the first time the $H^1$-norm of the solution explodes. Trivially, this occurs no later than the first jump, however, the mathematical challenge in attempting to restart the problem after an explosion time seems no less difficult than Conjecture \ref{Into_Conj_UniquenessNonVanish}.

\begin{thm}[Uniqueness up to explosion]
\label{Intro_Thm_MainUniqueness}
There exists a solution $L$ to (\ref{eq:Intro_MVproblem}) up to time
\[
t_\mathrm{explode} 
:= \sup\{
	t > 0: \Vert L \Vert_{H^1(0,t)} < \infty
\} \in (0,\infty]
\]
such that for every $t_0 < t_\mathrm{explode} $ we have $L \in \mathcal{S}(\tfrac{1-\beta}{2}, K, t_0)$ for some $K > 0$. Furthermore $L$ is the unique solution on $[0,t_\mathrm{explode})$: If $\bar{L}$ is any other generic c\`{a}dl\`{a}g solution to (\ref{eq:Intro_MVproblem}) satisfying \eqref{eq:Intro_PhysicalJumpCondition} on $[0,t]$ for some $t < t_\mathrm{explode}$, then $L_s = \bar{L}_s$ for all $s \in [0,t]$. 
\end{thm}

 The analogous result for the McKean--Vlasov problem corresponding to (\ref{eq:Intro_FiniteSystem_Def}) is presented in Section \ref{Sect_GeneralCoeff}. The first step towards uniqueness is given by Theorem \ref{Intro_Thm_Minimality}: The differentiable solution constructed in the first part of Theorem \ref{Intro_Thm_MainUniqueness} is minimal, so it must be a lower bound of any generic candidate c\`{a}dl\`{a}g solution, at least in small time. To obtain an upper bound, we introduce a family of processes in which we kill a small proportion, $\varepsilon > 0$, of the initial condition at time zero. In Section \ref{Sect_Bootstrap} we show that these modified solutions cannot overlap and so they bound the generic solutions from above. By returning to the contraction argument in Theorem \ref{Intro_Thm_Minimality}, we show that these modified solutions converge to the differentiable solution as the amount of deleted mass, $\varepsilon$, tends to zero. Thus, the envelope of solutions shrinks to zero size and this forces uniqueness of solutions. The power of this method is that it circumvents the need to have quantitative information about the generic candidate c\`{a}dl\`{a}g solutions.

\begin{rem}[Propagation of chaos]
Theorem \ref{Intro_Thm_MainUniqueness} resolves ambiguity about the validity of propagation of chaos. By following the methods in \cite{dirt_SPA_2015}, it can be shown that the particle system in (\ref{eq:Intro_FiniteSystem_Def}) has limit points for convergence in law (with respect to a suitable topology) to a c\`{a}dl\`{a}g solution of (\ref{eq:Intro_MVproblem}) satisfying \eqref{eq:Intro_PhysicalJumpCondition}. Now that we have uniqueness amongst general c\`{a}dl\`{a}g solutions, we can conclude that this is in fact full convergence, up to the first explosion time, and not just subsequential convergence.

\end{rem}

\subsubsection*{Overview of the paper}

In Section \ref{Sect_MinimalJmmps} we motivate and explain the physical  condition (\ref{eq:Intro_PhysicalJumpCondition}) and prove that it is necessary in order to have a c\`{a}dl\`{a}g solution. In Section \ref{Sect_Unique} we prove Theorem \ref{Intro_Thm_Minimality} via a comparison argument. In Section \ref{Sect_FixedPoint} we prove Theorem \ref{Intro_Thm_FixedPoint} by finding a space on which the map $\Gamma$ stabilises. In Section \ref{Sect_Bootstrap} we prove Theorem \ref{Intro_Thm_MainUniqueness} by extending the fixed point argument up to the explosion time and introducing the $\varepsilon$-deleted solutions used to bound candidate solutions. In Section \ref{Sect_GeneralCoeff} we show how our arguments can be extended to incorporate more general drift and diffusion coefficients.

\section{Minimal jumps and $\alpha$-fragility --- proof of Proposition \ref{Intro_Prop_MinimalJumps}}
\label{Sect_MinimalJmmps}
Our aims in this section are twofold: Firstly, we prove that the physical jump condition in (\ref{eq:Intro_PhysicalJumpCondition}) yields the solution with the smallest possible jump size (at any given instance) and, secondly, we provide some intuition for this behaviour. 

We begin by observing that regardless of the initial condition, at any time $t>0$, the measure $\nu_t$ will have a density. 

\begin{prop}[Existence of a density process]
\label{MinimalJumps_Prop_DensityExists}
	Let $(\nu_t)_{t \geq 0}$ be the law of a solution, $X$, to (\ref{eq:Intro_MVproblem}). For all $t > 0$, $\nu_t$ is absolutely continuous with respect to the Lebesgue measure and therefore has a density function, $V_t : (0, \infty) \to (0, \infty)$. Moreover, if $\nu_0$ has a density $V_{0}\in L^{\infty}$, then $\left\Vert V_{t}\right\Vert _{\infty}\leq\left\Vert V_{0}\right\Vert _{\infty}$.
\end{prop}

\begin{proof}
By omitting the killing at zero, we have the upper bound
\[
\nu_t(S) 
	\leq \mathbb{P}(X_0 + B_t -\alpha L_t \in S)
	= \int_S \int^\infty_0 p_t(x_0 + x - \alpha L_t) \nu_0(dx_0) dx,
\]
where $p_t$ is the Brownian transition kernel. Hence $\nu_{t}(S)\leq(2\pi t)^{-1/2}\text{Leb}(S)$, so the first claim follows from the Radon--Nikodym Theorem. If $\nu_0$ has a density $V_0$ in $L^{\infty}$, then we get
\[
V_{t}(x)\leq\int_{0}^{\infty}p_{t}(x_{0}+x-\alpha L_{t})V_{0}(x_{0})dx_{0}\leq\left\Vert V_{0}\right\Vert _{\infty},
\]
since $\int_{\mathbb{R}}p_{t}(y)dy=1$. This proves the second claim.
\end{proof}

As remarked in Section \ref{Sect_Intro}, equation (\ref{eq:Intro_JumpCondition}) alone is insufficient to determine the jump size at a given time. Not only is zero always a solution of that equation, but there may indeed be many other solutions that differ from the physical solution given by  (\ref{eq:Intro_PhysicalJumpCondition}). 

\begin{example}
\label{MinimalJumps_Ex_SomeV0}
Suppose we have an initial law $\nu_0$ with density
\[
V_0(x) = \alpha^{-1} \mathbf{1}_{0 < x < \alpha} + 2\alpha^{-1}\mathbf{1}_{2 \alpha < x < 3 \alpha}.
\]
Then any jump size $\Delta L_0 \in [0,1] \cup \{ 3 \}$ solves (\ref{eq:Intro_JumpCondition}), however, the physical solution is given by the condition $\Delta L_0 = 1$, see Figure \ref{fig:FixedPoint_MinJumpExample}.
\end{example}

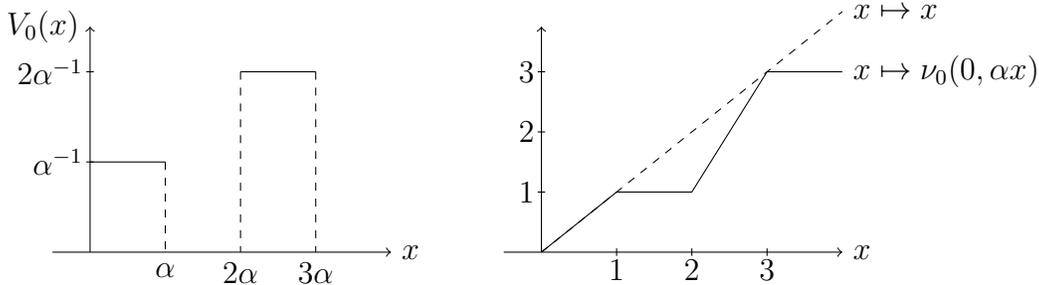
\begin{figure}
\begin{center}
\begin{tikzpicture}
\draw[->] (-0.5,0) -- (4,0) node [right] {$x$};
\draw[->] (0,0) -- (0,3) node [left] {$V_{0}(x)$};
\draw[-] (0,1.2) -- (1,1.2);
\draw[] (-0.05,1.2)--(0.05,1.2) node [left] {$\alpha^{-1}$};
\draw[-, dashed] (1,1.2) -- (1, 0) node [below] {$\alpha$};
\draw[-, dashed] (2,2.4) -- (2,0) node [below] {$2\alpha$};
\draw[-] (2,2.4) -- (3,2.4);
\draw[-, dashed] (3,2.4) -- (3,0) node [below] {$3\alpha$};
\draw[] (-0.05,2.4)--(0.05,2.4) node [left] {$2 \alpha^{-1}$};

\draw[->] (-0.5+6,0) -- (4+6,0) node [right] {$x$};
\draw[->] (0+6,0) -- (0+6,3) ;
\draw[-] (0+6,0) -- (1+6,0.8*1) -- (2+6,0.8*1) -- (3+6,0.8*3) -- (4+6,0.8*3) node[right] {$x \mapsto \nu_0(0,\alpha x)$};
\draw[-, dashed] (0+6,0) -- (3+6,0.8*3) -- (4+6, 0.8*4) node[right] {$x \mapsto x$};

\draw[] (1+6,-0.05)--(1+6,0.05) node [below] {$1$};
\draw[] (2+6,-0.05)--(2+6,0.05) node [below] {$2 $};
\draw[] (3+6,-0.05)--(3+6,0.05) node [below] {$3 $};
\draw[] (0+6-0.05,0.8)--(0+6+0.05,0.8) node[left] {$1$};
\draw[] (0+6-0.05,0.8*2)--(0+6+0.05,0.8*2) node[left] {$2$};
\draw[] (0+6-0.05,0.8*3)--(0+6+0.05,0.8*3) node[left] {$3$};

\end{tikzpicture}
\caption{\label{fig:FixedPoint_MinJumpExample} The function from Example \ref{MinimalJumps_Ex_SomeV0} is on the left. The candidate jumps --- that is, solutions to (\ref{eq:Intro_JumpCondition}) --- are the points on the right where the graphs intersect. The point $(1 ,1)$ gives the minimal allowed jump size of $1$, since $x = 1$ is the value given by (\ref{eq:Intro_PhysicalJumpCondition}). }
\end{center} \vspace{-10pt}
\end{figure}

An alternative way to view (\ref{eq:Intro_PhysicalJumpCondition}), which better motivates Proposition \ref{Intro_Prop_MinimalJumps}, is the notion of \emph{fragility}. To state a proper definition, consider first the sequence 
\begin{align}
\label{eq:MinimalJumps_FragileSequence}
f_0(\mu,\alpha,\varepsilon) &:= \mu(0,\alpha\varepsilon)\nonumber \\ 
f_{n+1}(\mu,\alpha,\varepsilon) &:= \mu(0,\alpha\varepsilon + \alpha f_n (\mu,\alpha,\varepsilon)), 	\qquad \textrm{for }n \geq 0,
\end{align}
for a given measure $\mu$ and constants $\varepsilon > 0$ and $\alpha > 0$. Notice that $f_{n}(\mu,\alpha,\varepsilon)$ is increasing in $n$ and in $\varepsilon$. Hence we can deduce that the limit $f_\infty(\mu,\alpha,\varepsilon)$ exists and is an increasing function of $\varepsilon$. Therefore, we can sensibly define:

\begin{defn}[$\alpha$-fragility]
The measure, $\mu$, is said to be \emph{$\alpha$-fragile} if 
\[
\lim_{\varepsilon \to 0}f_\infty(\mu,\alpha,\varepsilon) \neq 0.
\]
\end{defn}

To see why $\alpha$-fragility is related to physical solutions, consider starting the heat equation from an initial law $\nu_0$. In small time, we will immediately lose mass at the boundary, say an $\varepsilon$ amount. To approximate the contagious system in (\ref{eq:Intro_MVproblem}), we must then shift the measure down towards the origin by an amount $\alpha \varepsilon$. If we apply this to $\nu_0$, we obtain a loss of $f_0(\nu_0, \varepsilon, \alpha)$, which we can then further shift our initial condition by, and so on, hence obtaining 
$f_\infty(\nu_0,\alpha,\varepsilon)$ 
in the limit. If this terminal quantity does not shrink to zero with $\varepsilon$, we should not expect the solution to (\ref{eq:Intro_MVproblem}) to be right-continuous at $t = 0$: An infinitesimal loss of mass starts a cascade of losses summing to a non-zero amount. Indeed, this is a heuristic version of the argument presented in the proof of Proposition \ref{Intro_Prop_MinimalJumps} below. For now, we notice that there is, in fact, an exact correspondence between $f_\infty(\mu,\alpha,0+)$ and the physical jump condition.

\begin{prop}
For any atomless measure $\mu$ on the positive half-line and $\alpha > 0$ we have
\[
\lim_{\varepsilon \to 0} f_\infty(\mu, \alpha, \varepsilon) 
	= \inf\{ x \geq 0 : \mu(0,\alpha x)  < x\}.
\]
\end{prop}

\begin{proof}
Write $f = \lim_{\varepsilon \to 0} f_\infty(\mu, \alpha, \varepsilon)$ and $x_0 = \inf\{ x \geq 0 : \mu(0,\alpha x)  < x\}$. 
 
Suppose $f < x_0$. Then we can find $\varepsilon > 0$ such that $f_\infty(\mu, \alpha, \varepsilon) + \varepsilon < x_0$. By taking the limit as $n \to \infty$ in (\ref{eq:MinimalJumps_FragileSequence}) we get
\[
f_\infty(\mu, \alpha, \varepsilon) 
	= \mu(0, \alpha \varepsilon + \alpha f_\infty(\mu, \alpha, \varepsilon) )
	\geq \varepsilon + f_\infty(\mu, \alpha, \varepsilon),
\]
where the inequality is due to the definition of $x_0$, but this is clearly a contradiction. 

Suppose $f > x_0$. By definition of $x_0$, we can find a sequence $(x_\delta) \downarrow x_0$ as $\delta \downarrow 0$ such that
\[
\mu(0,\alpha x_\delta ) \leq x_\delta - \delta.
\]
Fix any such $\delta > 0$ and take $\varepsilon\in (0, x_\delta \wedge \delta)$, then we have
\[
f_0(\mu,\alpha,\varepsilon) 
	= \mu(0,\alpha \varepsilon) 
	\leq \mu(0,\alpha x_\delta) 
	\leq x_\delta - \delta
	\leq x_\delta - \varepsilon.
\]
Since $f_0(\mu,\alpha,\varepsilon) + \varepsilon \leq x_\delta$ we may now repeat the argument to get
\[
f_1(\mu,\alpha,\varepsilon) 
	= \mu(0,\alpha \varepsilon + \alpha f_0(\mu,\alpha,\varepsilon)) 
	\leq \mu(0,\alpha x_\delta) 
	\leq x_\delta - \delta
	\leq x_\delta - \varepsilon,
\]
and so on. Proceeding inductively, we conclude that $f \leq x_\delta - \delta$, so taking $\delta\downarrow0$ gives $f \leq x_0$, which is a contradiction.  
\end{proof}

\begin{example}
Returning to the density $V_0$ from Example \ref{MinimalJumps_Ex_SomeV0}, set $N_\varepsilon:=\left\lfloor \varepsilon^{-1}-1\right\rfloor $ for $\varepsilon <1$. Then $\alpha (n+1) \varepsilon \leq \alpha$ for all $n\leq N_\varepsilon$, so we get
\[
f_n(V_0,\alpha, \varepsilon) = (n+1)\varepsilon \leq 1
	\qquad \forall n\leq N_\varepsilon.
\]
Thereafter, $\alpha < \alpha (N_\varepsilon+2)\varepsilon < 2 \alpha$, so $f_{N_\varepsilon+1}(V_0,\alpha, \varepsilon)=1$ and, since $\alpha < \alpha \varepsilon + \alpha < 2 \alpha$, it follows that $f_n(V_0,\alpha, \varepsilon)$ remains equal to $1$ for all $n > N_\varepsilon$. Consequently, we have $f_\infty(V_0,\alpha, \varepsilon) = 1$ for all $\varepsilon <1$, and hence $f_\infty(V_0,\alpha, 0+) = 1$, which indeed agrees with the physical jump condition. 
\end{example}

\begin{rem}[Restarting solutions at a non-zero time]
\label{MinimalJumps_Rem_StartingNotZero}
To prove Proposition \ref{Intro_Prop_MinimalJumps} it will be sufficient to take $t = 0$ in the statement of the result, since we can always restart the system at a time $t > 0$, taking $t$ as our new time origin. This might be confusing at first glance as we specified $L_0 = 0$ in our definition of a solution to (\ref{eq:Intro_MVproblem}). To be clear, if we want to solve (\ref{eq:Intro_MVproblem}) from a time $t > 0$ onwards, and we have a solution $L$ up to time $t$, then we can solve the problem 
\[
\begin{cases}
\widetilde{X}_{s}^{x}=x+B_{s}-\alpha\widetilde{L}_{s}\\
\widetilde{\tau}^{x}=\inf\{s\geq0:\widetilde{X}_{s}^{x}\leq0\}\\
\widetilde{L}_{s}=\int_{0}^{\infty}\mathbb{P}(\widetilde{\tau}^{x}\leq s)V_{t}(x)dx\\
\widetilde{L}_{0}=0\textrm{ and }\widetilde{L}\textrm{ c\`{a}dl\`{a}g},
\end{cases}
\]
for $s > 0$, where $V_t$ is the density for the problem at time $t$. Indeed, this is exactly the same formulation as (\ref{eq:Intro_MVproblem}), except the initial condition is a sub-probability density. We then have that
\[
L_u := L_t + \widetilde{L}_{u-t}
\]
is an extension of $L$ for $u \geq t$, which solves (\ref{eq:Intro_MVproblem}). Therefore, restarting the system at a non-zero time is just a matter of normalising the initial condition. 
\end{rem}

\begin{proof}[Proof of Proposition \ref{Intro_Prop_MinimalJumps}]
	For a contradiction, suppose that $L$ is a c\`adl\`ag process satisfying (\ref{eq:Intro_MVproblem}) for which the inequality in Proposition \ref{Intro_Prop_MinimalJumps} is violated. As noted in Remark \ref{MinimalJumps_Rem_StartingNotZero}, it suffices to assume that this occurs at $t = 0$, and this implies that $x_0 := \inf\{ x > 0 : \nu_0(0,\alpha x) < x\} > 0$ and $ L_h \to 0$ as $h \downarrow 0$. By definition of $x_0$ we have
\[
\nu_0(0,x) \geq \alpha^{-1} x
	\quad \textrm{for}\quad x < \alpha x_0.
\] 
Thus, if $F$ is a decreasing differentiable function, then
\begin{align*}
\int_0^{\alpha x_0} F(x)\nu_0(dx)
	&= F(\alpha x_0)\nu_0(0,\alpha x_0) - \int^{\alpha x_0}_0 F'(x)\nu_0(0,x)dx \\ 
	&\geq \alpha^{-1} \Big( \alpha x_0 F(\alpha x_0) - \int^{\alpha x_0}_0 x F'(x)  dx \Big) \\
	&= \alpha^{-1} \int_0^{\alpha x_0} F(x)dx. 
\end{align*}
Using this lower bound, it holds for any $h>0$ that
\begin{align*}
\ L_h 
	= \mathbb{P}(\tau \leq h) 
	&\geq \int^\infty_0 \mathbb{P}(B_h  \leq \alpha L_h-x)\nu_0(dx) \\
	&\geq \alpha^{-1} \int^{\alpha x_0}_0 \mathbb{P}(B_h - \alpha L_h \leq -x) dx \\
	&= \alpha^{-1} h^{1/2} \int^{\alpha h^{-1/2}( x_0 - L_h)}_{- \alpha h^{-1/2} L_h} \Phi(-y)dy.
\end{align*}
Here $\Phi$ is the Normal c.d.f.~and we shall also need the Normal p.d.f., $\phi$. Observe that the function
\[
\Psi(x) := \phi(x) - x\Phi(-x), \qquad x \in \mathbb{R}
\]
satisfies
\begin{equation}
\label{eq:MinimalJumps_ThePsiFunction}
\Psi'(x) = - \Phi(-x), \qquad \textrm{for every } x \in \mathbb{R}.
\end{equation}

Now let $c_0:=\alpha x_0 /2$. Since we are assuming $L_h \to 0$, we have $\alpha (x_0-L_h)>c_0$ for all $h>0$ sufficiently small, and hence
\[
\alpha h^{-1/2} L_h
	\geq \int^{c_0 h^{-1/2} }_{- \alpha h^{-1/2} L_h} \Phi(-y)dy = \bigl[x\Phi(-x)-\phi(x)\bigr]_{-\alpha h^{-1/2}L_{h}}^{c_{0}h^{-1/2}}.
\] 
Using $\Phi(\alpha h^{-1/2} L_h)=1-\Phi(-\alpha h^{-1/2} L_h)$ and $\phi(-\alpha h^{-1/2} L_h)=\phi(\alpha h^{-1/2} L_h)$, we can rearrange this inequality to find that
\[
\phi(\alpha h^{-1/2} L_h) - \alpha h^{-1/2} L_h \Phi(-\alpha h^{-1/2} L_h)   
	\leq \phi(c_0 h^{-1/2} ) -c_0  h^{-1/2}   \Phi(-c_0 h^{-1/2}).
\]
In other words, for all $h > 0$ sufficiently small, we have
\[
\Psi(\alpha h^{-1/2} L_h) \leq \Psi(c_0 h^{-1/2}),
\]
but this is our required contradiction: By (\ref{eq:MinimalJumps_ThePsiFunction}) $\Psi$ is strictly decreasing and if $L_h \to 0$ then we can certainly find $h$ sufficiently small so that $\alpha h^{-1/2} L_h  < c_0 h^{-1/2}$.  
\end{proof}

\section{Contractivity of $\Gamma$ and minimality of differentiable solutions --- proof of Theorem \ref{Intro_Thm_Minimality}}
\label{Sect_Unique}

The main objective in this section is to prove Theorem \ref{Intro_Thm_Minimality}, which states that differentiable solutions are minimal and that $\Gamma$ is contractive on $L^{\infty}$. The key technique is the following comparison result, which is obtained by coupling two outputs, $\Gamma[\ell]$ and $\Gamma[\bar{\ell}]$, of the map (\ref{eq:Intro_DefnOfGamma}) to the same driving Brownian motion.

\begin{prop}[Comparison]
\label{Unique_Prop_Comparison}
Let $\ell$ and $\bar{\ell}$ be two increasing and initially zero c\`{a}dl\`{a}g functions and $\Gamma$ be defined as in (\ref{eq:Intro_DefnOfGamma}) for model parameters $\alpha$ and $\nu_0$. Suppose that $\ell$ is continuous on $[0,t_0)$. Then
\[
(\Gamma[\ell]_t - \Gamma[\bar{\ell}]_t)^+
	\leq \int^t_0 \Big\{ 2 \Phi \Big( \alpha \frac{(\ell_s -\bar{\ell}_s)^+}{\sqrt{t-s}} \Big) - 1\Big\} d\Gamma[\ell]_s,
	\qquad \textrm{for every } t < t_0,
\]
where $x^+ := \max\{ x,0\}$ and $\Phi$ is the Normal c.d.f.
\end{prop}

\begin{proof}
It is no loss of generality to take a Brownian motion, $B$, and initial value, $X_0$, such that (\ref{eq:Intro_DefnOfGamma}) holds with this $B$ and this $X_0$ for both $\ell$ and $\bar{\ell}$, that is
\[
X^\ell_t = X_0 + B_t -\alpha \ell_t
\qquad \textrm{and}\qquad
X^{\bar{\ell}}_t = X_0 + B_t  - \alpha \bar{\ell}_t,
\]
where we denote the respective hitting times of zero by $\tau$ and $\bar{\tau}$.

By conditioning on the value of $\tau$ we have 
\begin{align}
\label{eq:Unique_MainMinimalArgument}
\Gamma[\ell]_t - \Gamma[\bar{\ell}]_t
	&\leq \mathbb{P}(\tau \leq t, \bar{\tau} > t) \\
	&= \int^t_0 \mathbb{P}(\bar{\tau} > t | \tau = s)d\Gamma[\ell]_s \nonumber\\
	&= \int^t_0 \mathbb{P}(\inf_{u \in [0,t]} X^{\bar{\ell}}_u > 0 | \tau = s)d\Gamma[\ell]_s \nonumber\\
	&\leq \int^t_0 \mathbb{P}(\inf_{u \in [s,t]} \{ X^{\bar{\ell}}_s + B_u - B_s - \alpha(\bar{\ell}_u - \bar{\ell}_s) \} > 0 | \tau = s) d\Gamma[\ell]_s,\nonumber
\end{align}
where in the final line we have discarded the contribution on $u \in [0,s]$. Since $\bar{\ell}$ is increasing and $\alpha \geq 0$ we can further bound (\ref{eq:Unique_MainMinimalArgument}) by
\[
\Gamma[\ell]_t - \Gamma[\bar{\ell}]_t 
	\leq \int^t_0 \mathbb{P}(\inf_{u \in [s,t]} \{ X^{\bar{\ell}}_s + B_u - B_s \} > 0 | \tau = s) d\Gamma[\ell]_s. 
\]
On the event $\{\tau = s\}$ we have $X^\ell_s = 0$, so
\[
X^{\bar{\ell}}_s = X^{\bar{\ell}}_s - X^\ell_s = \alpha(\ell_s - \bar{\ell}_s)
\]
and thus we have 
\begin{align*}
\Gamma[\ell]_t - \Gamma[\bar{\ell}]_t 
	&\leq \int^t_0 \mathbb{P}(\inf_{u \in [s,t]} \{B_u - B_s \} >- \alpha(\ell_s - \bar{\ell}_s)) d\Gamma[\ell]_s \\
	&\leq  \int^t_0 \mathbb{P}(\inf_{u \in [s,t]} \{B_u - B_s \} > -\alpha(\ell_s - \bar{\ell}_s)^+)  d\Gamma[\ell]_s\\
	&= \int^t_0 \Big\{ 2 \Phi \Big( \alpha \frac{(\ell_s - \bar{\ell}_s)^+}{\sqrt{t-s}} \Big) - 1\Big\} d\Gamma[\ell]_s.
\end{align*}
Since the right-hand side is positive, we can replace the left-hand side by its maximum with zero, so the proof is complete.
\end{proof}

When $\Gamma[\ell]$ is differentiable with power law control on its derivative near zero, then we are able to use this derivative to get a more direct bound.

\begin{cor}
\label{Unique_Cor_DifferentialComparison}
Suppose that $\Gamma[\ell] \in \mathcal{S}(\gamma, A, t_0)$ for some $\gamma \in (0,\tfrac{1}{2})$, $A > 0$ and $t_0 > 0$ (recall Definition \ref{Intro_Def_SetS}). Then there exists $c_0 > 0$, independent of $t_0$ and $\ell$, such that
\[
(\Gamma[\ell]_t - \Gamma[\bar{\ell}]_t)^+ 
	\leq c_0 \int^t_0 \frac{(\ell_s - \bar{\ell}_s)^+}{ (t-s)^{\frac{1}{2}}s^{\gamma}} ds,  
		\qquad \textrm{for every } t \leq t_0,
\]
where $\bar{\ell}$ is any c\`{a}dl\`{a}g function that is increasing and initially zero.
\end{cor}

\begin{proof}
We begin by noticing that $x \mapsto 2\Phi(\alpha x) - 1$ is bounded above on $x \geq 0$ by the linear function 
\[
x \mapsto c_1 x := \alpha \sqrt{2/\pi} \, x.
\]
Therefore, by Proposition \ref{Unique_Prop_Comparison},
\[
(\Gamma[\ell]_t - \Gamma[\bar{\ell}]_t)^+ 
	\leq  c_1  \int^t_0 \frac{(\ell_s - \bar{\ell}_s)^+}{(t-s)^{\frac{1}{2}}} \Gamma[\ell]_s' ds.
\]
The result follows since $\Gamma[\ell]_s' \leq As^{-\gamma}$, by definition of $\mathcal{S}(\gamma,A,t_0)$. 
\end{proof}

It is now a straightforward task to prove Theorem \ref{Intro_Thm_Minimality}. 

\begin{proof}[Proof of Theorem \ref{Intro_Thm_Minimality}]
With $A$ and $\gamma$ fixed as in the statement of the result, take $\ell, \bar{\ell} \in \mathcal{S}(\gamma, A, t_0)$, where we will show how to take $t_0$ sufficiently small (and independent of $\ell$ and $\bar{\ell}$) so that we have the result. With $c_0 > 0$ as in Corollary \ref{Unique_Cor_DifferentialComparison} and the symmetry in $\ell \leftrightarrow \bar{\ell}$, we have
\begin{align*}
|\Gamma[\ell]_t - \Gamma[\bar{\ell}]_t|
	\leq c_0 \int^t_0 \frac{|\ell_s - \bar{\ell}_s|}{(t-s)^\frac{1}{2} s^{\gamma}}  ds
	&\leq c_0 \Vert \ell -\bar{\ell} \Vert_{L^\infty(0,t_0)} \int^t_0 \frac{ds}{(t-s)^\frac{1}{2} s^{\gamma}} \\
	&\leq c_1 t_0^{\frac{1}{2}-\gamma}\Vert \ell -\bar{\ell} \Vert_{L^\infty(0,t_0)},
\end{align*}
for $t < t_0$, where $c_1 > 0$ is a constant independent of $t_0$, $\ell$ and $\bar{\ell}$. Consequently, it suffices to take $t_0$ such that $c_1 t_0^{1/2 - \gamma} \leq \tfrac{1}{2}$.

For the second half of the result, take $t_0$ now fixed as in the statement. By the same estimate as above applied to Corollary \ref{Unique_Cor_DifferentialComparison}, we have
\[
(L_t - \bar{L}_t)^+
	\leq c_1 t_1^{\frac{1}{2}-\gamma} \Vert (L - \bar{L})^+ \Vert_{L^\infty(0,t_1)},
		\qquad \textrm{for every } t \leq t_1,
\]
where $t_1$ is any time with $t_1 < t_0$. Therefore, taking $t_1 > 0$ such that $c_1 t_1^{\frac{1}{2}-\gamma} < 1$ forces
\[
\Vert (L - \bar{L})^+ \Vert_{L^\infty(0,t_1)} = 0,
\]
that is $\bar{L}_t \geq L_t$ for $t \leq t_1$. 

By returning to Corollary \ref{Unique_Cor_DifferentialComparison}, and repeating the same estimate again, we can deduce that
\[
(L_t - \bar{L}_t)^+ 
	\leq \Vert (L - \bar{L})^+ \Vert_{L^\infty(t_1, t_2)} \int^{t}_{t_1} \frac{ds}{(t-s)^\frac{1}{2} s^\gamma} ds,
		\qquad \textrm{for every } t \leq t_2,
\]
for any $t_2 < t_0$. Thus, by taking $t_2 > t_1$ sufficiently close to $t_1$, we can again force 
\[
\Vert (L - \bar{L})^+ \Vert_{L^\infty(0,t_2)} = 0.
\]
Continuing to repeat this argument, we obtain a sequence of times $t_1 < t_2 < \cdots < t_n < \cdots < t_0$. If $t_n \to t_\infty < t_0$, then the argument also applies at time $t_\infty$ and so can be restarted. Furthermore, if this procedure ever terminates at a time strictly less than $t_0$, then the argument can be restarted (by left continuity) for a non-zero time, thus contradicting the termination. Hence we conclude that we have uniqueness up to the time $t_0$.  
\end{proof}

\section{Stability of $\Gamma$ and the fixed-point argument --- proof of Theorem \ref{Intro_Thm_FixedPoint}}
\label{Sect_FixedPoint}

Here we show that we can choose parameters such that $\Gamma$ maps $\mathcal{S}(\gamma, A, t_0)$ into itself. From this and Theorem \ref{Intro_Thm_Minimality}, we are then able to conclude the existence of a solution in short time by the Banach fixed point theorem. We will also carefully track the contribution to the regularity of the solution near time zero due to the decay of the initial density near the Dirichlet boundary. To this end, we will see that the exponent $\frac{1-\beta}{2}$ in Theorem \ref{Intro_Thm_FixedPoint} comes directly from Lemma \ref{Fixed_Lem_BoundTheBits} below.

\begin{lem}
\label{Fixed_Lem_StartingPoint}
With $\ell$ and $\Gamma$ above
\[
\Gamma[\ell]_t = 1-  \int^\infty_0 \int^\infty_0 G_{t}(x_0, x) dx \, \nu_0(dx_0) 
	+2 \alpha \int^t_0  \mathbb{E}[ 
		p_{t-s}(X_s)\mathbf{1}_{s < \tau}
	] d\ell_s,
\]
where
\[
p_{t}(x) := \frac{1}{\sqrt{2\pi t}} e^{-\frac{x^2}{2t}},
	\qquad G_t(x,y) := p_t(x - y) - p_t(x+y).
\]
\end{lem} 

\begin{proof}
Fix $t$ and take $u$ to be the solution to the backward heat equation on the positive half-line with Dirichlet boundary condition:
\[
\partial_t u(s,x) + \tfrac{1}{2} \partial_{xx} u(s,x) = 0, \qquad u(\cdot, 0) =0 ,\; u(t, x) = \mathbf{1}_{x > 0}.
\]
That is, let
\[
u(s,x) := \int^\infty_0 G_{t-s}(x,y)dy. 
\]
By applying It\^o's formula to $s \mapsto u(s, X_{s \wedge \tau})$ we obtain
\[
du(s,X_{s \wedge \tau})
	= \mathbf{1}_{s < \tau}\{ \partial_s + \tfrac{1}{2} \partial_{xx} \} u(s,X_s) ds
	+ \mathbf{1}_{s < \tau}\partial_x u(s, X_s) d(B_s - \alpha \ell_s).
\]
Since $\partial_s u+ \tfrac{1}{2} \partial_{xx}u=0$, integrating in time and taking expectation gives 
\[
\mathbb{E}[u(t,X_{t \wedge \tau})] 
	= \mathbb{E}[u(0, X_0)]
	- \alpha \int^t_0 \mathbb{E} [ \mathbf{1}_{s < \tau}\partial_x u(s, X_s) ] d\ell_s,
\]
where the endpoints take the values
\begin{gather*}
\mathbb{E}[u(t,X_{t \wedge \tau})] = \mathbb{P}(\tau > t) = 1 - \Gamma[\ell]_t \\
\mathbb{E}[u(0, X_0)] = \mathbb{E}\int^\infty_0 G_{t}(X_0, x) dx = \int^\infty_0 \int^\infty_0 G_{t}(x_0, x) dx \, \nu_0(dx_0) .
\end{gather*}
It remains to notice that 
\begin{align*}
\partial_x u(s, x) 
	&= \int^\infty_0 \frac{1}{\sqrt{2\pi (t-s)}}\partial_x \Big\{ e^{-\frac{(x-y)^2}{2(t-s)}} - e^{-\frac{(x+y)^2}{2(t-s)}} \Big\} dy \\
	&= -\int^\infty_0 \frac{1}{\sqrt{2\pi (t-s)}} \partial_y \Big\{ e^{-\frac{(x-y)^2}{2(t-s)}} + e^{-\frac{(x+y)^2}{2(t-s)}} \Big\} dy \\
	&= \frac{2}{\sqrt{2\pi (t-s)}} e^{-\frac{x^2}{2(t-s)}} = 2p_{t-s}(x). 
\end{align*}
\end{proof}

Our strategy is to apply the method of difference quotients \cite[Sect.~5.8.2, Thm.~3]{evans_2010} using Lemma \ref{Fixed_Lem_StartingPoint}. With $t, \delta > 0$ fixed, the starting point is to write
\begin{align}
\label{eq:Fixed_SplitUp}
\Gamma[\ell]_{t+\delta} - \Gamma[\ell]_t \nonumber
	&= -\int^\infty_0 \int^\infty_0  (G_{t+\delta}(x_0,x) - G_t(x_0,x)) dx \nu_0(dx_0) \\
	&\qquad+ 2\alpha\int^{t+\delta}_t \mathbb{E}[p_{t+\delta-s}(X_s)\mathbf{1}_{s < \tau}] d\ell_s \nonumber \\
	&\qquad+ 2\alpha \int^t_0 \mathbb{E}[(p_{t+\delta-s}(X_s)-p_{t-s}(X_s))\mathbf{1}_{s < \tau}]d\ell_s  \nonumber \\
	&=: I_1 + I_2 + I_3. 
\end{align}
In order to estimate these three integrals we must make use of the assumption (\ref{eq:Intro_DensityAssumption}) on the behaviour of $\nu_0$ near the origin. Recall that we have constants $C, D, x_\star, \beta  > 0$ such that 
\[
V_0(x) \leq C x^\beta \mathbf{1}_{x \leq x_\star} + D \mathbf{1}_{x > x_\star}. 
\]

\begin{lem}[$I_1$]
\label{Fixed_Lem_BoundTheBits}
There exists $t_0 = t_0(C, D, x_0)$ and $K = K(C, x_0)$ such that
\begin{align*}
0 \leq I_1 \leq K \delta t^{-\frac{1-\beta}{2}},
\end{align*}
for all $t < t_0$ and $\delta > 0$. 
\end{lem}

\begin{rem}[Constants]
In the proof below we will allow  the constants to increase as necessary, but we will pay close attention to ensure that $K$ does not depend on $D$ (whereas $t_0$ does).
\end{rem}

\begin{proof}[Proof of Lemma \ref{Fixed_Lem_BoundTheBits}]
By the fundamental theorem of calculus, and properties of $G$, we have
\begin{align*}
I_1 &= -\delta \int^\infty_0 \int^\infty_0 \int_0^1 \partial_t G_{t + \theta \delta}(z,x)d\theta dx \nu_0(dz) \\
	&= - \frac{\delta}{2} \int^\infty_0 \int^\infty_0 \int_0^1 \partial_{xx} G_{t + \theta \delta}(z,x)d\theta dx \nu_0(dz) \\
	&= \frac{\delta}{2} \int^1_0 \int^\infty_0 \partial_x G_{t + \theta \delta}(0,z) \nu_0(dz) d\theta \\
	&= \frac{\delta}{\sqrt{2\pi }} \int^1_0 \int^\infty_0 \frac{z}{(t + \theta \delta)^{3/2}}e^{-\frac{z^2}{2(t + \theta \delta)}} V_0(z)dz d \theta \\
	&= \frac{\delta}{\sqrt{2\pi }} \int^1_0 \int^\infty_0 \frac{z}{(t + \theta \delta)^{1/2}}e^{-\frac{z^2}{2}} V_0((t + \theta \delta)^{\frac{1}{2}}z)dz d\theta.
\end{align*}
Splitting the $z$-integral and using the bound on $V_0$ gives
\begin{align*}
I_1 &\leq \frac{\delta }{\sqrt{2\pi }} \int^1_0 \Big( C \int^{x_\star(t + \theta \delta)^{-\frac{1}{2}}}_0 (t + \theta \delta)^{- \frac{1-\beta}{2}  } z^{1+\beta} e^{-\frac{z^2}{2}} dz \\ 
&\qquad\qquad\qquad\qquad\qquad + D \int_{x_\star(t + \theta \delta)^{-\frac{1}{2}}}^\infty (t + \theta \delta)^{-\frac{1}{2}} z e^{-\frac{z^2}{2}} dz
 \Big) d\theta  \\
  &\leq \frac{\delta }{\sqrt{2\pi }} \Big( K (t+ \delta)^{- \frac{1-\beta}{2}} + D(t + \delta)^{-\frac{1}{2}}e^{-\frac{x_\star^2}{2(t + \delta)}} \Big).
\end{align*}
Note that the final term is a bounded function of $D$ and $x_\star$, so by taking $t_0$ sufficiently small we can conclude the result. 
\end{proof}

In order to control $I_2$ from equation (\ref{eq:Fixed_SplitUp}) we need to make some assumptions on the regularity of $\ell$. Provided $\ell$ is in $H^1$, we have that $t \mapsto B_t - \alpha \ell_t$ is absolutely continuous with respect to Brownian motion, and so we can proceed as below with a change of measure. 

\begin{lem}[A Girsanov argument]
\label{Fixed_Lem_Girsanov}
Let $F : \mathbb{R} \to \mathbb{R}$ be measurable and bounded and assume that $ \ell \in H^1(0,t_0)$ for some $t_0 > 0$.  Then for all $q > 1$ and $t \in [0,t_0]$
\[
\mathbb{E}[F(X_t)\mathbf{1}_{t < \tau}] 
	\leq \exp\Big\{ \frac{\alpha^2 \Vert \ell' \Vert^2_{L^2(0,t)} }{2(q-1)^2}   \Big\} \mathbb{E}[F(W_t)^q\mathbf{1}_{t < \tau^W}]^{\frac{1}{q}}  
\]
where $W$ is a standard Brownian motion started from $\nu_0$ and $\tau^W = \inf\{t\geq 0: W_t = 0\}$.
\end{lem}

\begin{proof}
Let $Z_t$ denote the Radon--Nikodym derivative
\[
Z_t 
	= \frac{d\mathbb{Q}}{d\mathbb{P}}\Big|_{\mathcal{F}_t}= \exp\Big\{ 
		\alpha \int^t_0 \ell_s' dB_s - \frac{\alpha^2}{2}  \int^t_0 (\ell_s')^2 ds
	\Big\}
\]
for $t \in [0, t_0]$. Under $\mathbb{Q}$, $t \mapsto B_t - \alpha \ell_t$ is a standard Brownian motion by Girsanov's Theorem. Therefore, applying H\"older's inequality gives
\begin{align*}
\mathbb{E}[F(X_t)\mathbf{1}_{t < \tau}]
	&= \mathbb{E}^{\mathbb{Q}} [Z_t^{-1} F(X_t)\mathbf{1}_{t < \tau}] \\
	&\leq \mathbb{E}^{\mathbb{Q}} [Z_t^{-r}]^{\frac{1}{r}} \mathbb{E}^{\mathbb{Q}} [F(X_t)^q\mathbf{1}_{t < \tau}]^{\frac{1}{q}} \\
	&\leq  \mathbb{E} [Z_t^{1-r}]^{\frac{1}{r}}  \mathbb{E} [F(W_t)^q \mathbf{1}_{t < \tau^W}]^{\frac{1}{q}},
\end{align*}
where $q^{-1}  + r^{-1} = 1$. Finally, it is a simple calculation to see that 
\[
\mathbb{E} [Z_t^{1-r}]^{\frac{1}{r}} 
	= \exp\Big\{ \frac{\alpha^2}{2}(r-1)^2 \Vert \ell' \Vert_{L^2(0,t)}^2 \Big\},
\]
and this completes the proof. 
\end{proof}

\begin{cor}
\label{Fixed_Cor_MainIntegralCaculation}
Suppose $\ell$ satisfies the conditions of Lemma \ref{Fixed_Lem_Girsanov} and fix $q > 1$. Then there is a constant $c_0>0$ depending only on $V_0$ and $q$ such that, for all $  u, s > 0$ and $\zeta ,\eta \geq 0$,
\[
\mathbb{E}[X_s^\zeta p_u(X_s) \mathbf{1}_{s < \tau}]
 \leq c_0 \exp\Big\{ \frac{\alpha^2 \Vert \ell' \Vert^2_{L^2(0,s)}}{2(q-1)^2}   \Big\} (u^{\frac{\beta}{2} + \eta + \frac{q\zeta}{ 2} - \frac{q-1}{2}}s^{-\eta} + u^{\eta + \frac{q\zeta}{2} - \frac{q-1}{2} } s^{\frac{\beta}{2} - \eta})^{\frac{1}{q}}.
\]
\end{cor}

\begin{proof}
We first prove the bound for Brownian motion, and then we appeal to the previous lemma for the full result. From the assumption (\ref{eq:Intro_DensityAssumption}) we can certainly find a constant $c_0 > 0$ such that $V_0(x) \leq c_0 x^\beta$ for all $x > 0$, and so 
\begin{equation*}
\mathbb{E}[W_s^{q\zeta} p_u(W_s)^q \mathbf{1}_{s < \tau^W}]
	\leq \frac{c_0}{u^{\frac{q}{2}}} \int^\infty_0 \int^\infty_0 x^{q\zeta} e^{-\frac{qx^2}{2u}} G_s(x_0, x) x_0^\beta dx_0 dx,
\end{equation*}
where from here on we absorb the relevant numerical constants into $c_0$. Using the estimate 
\[
G_s(x_0, x) \leq \frac{2}{(2\pi s)^{1/2}}\Big(\frac{xx_0}{s} \wedge 1 \Big) e^{-\frac{(x-x_0)^2}{2s}}
\]
together with the substitution $x_0 \mapsto x_0 - x$ gives 
\begin{equation}
\label{eq:Fixed_RequiredIntergal_I}
\mathbb{E}[W_s^{q\zeta} p_u(W_s)^q \mathbf{1}_{s < \tau^W}]
	\leq \frac{c_0}{u^{\frac{q}{2}}s^{\frac{1}{2}}} \int^\infty_0 x^{q\zeta} e^{-\frac{qx^2}{2u}} \int_{-x}^\infty (x + x_0)^\beta    \Big( \frac{x(x+x_0)}{s} \wedge 1 \Big) e^{-\frac{x_0^2}{ 2s}} dx_0 dx.
\end{equation}

Now notice that for generic $a, b \geq 0$ we have the inequality
\[
(a + b) \wedge 1 \leq (a \wedge 1) + (b \wedge 1).
\]
Hence it holds for $x \geq 0$, $x_0 \geq -x$ and $\eta \in [0,1]$ that 
\[
\frac{x(x+x_0)}{s} \wedge 1 
	\leq \Big( \frac{x^2}{s} \wedge 1 \Big) + \Big( \frac{|x_0x|}{s} \wedge 1 \Big)
	\leq \Big( \frac{x^2}{s} \Big)^\eta + \Big( \frac{|x_0x|}{s} \Big)^{2\eta}.
\]
Combining this inequality with $(x + x_0)^\beta \leq c_1( |x|^{\beta} + |x_0|^{\beta} )$, for $c_1 > 0$ a constant depending only on $\beta$, allows us to bound the inner integral in (\ref{eq:Fixed_RequiredIntergal_I}) by
\begin{align*}
&\int_{-x}^\infty (|x|^\beta+|x_0|^\beta)\Big( \Big( \frac{x^2}{s} \Big)^\eta + \Big( \frac{|x_0x|}{s} \Big)^{2\eta}\Big) e^{-\frac{x_0^2}{2s}} dx_0 \\
&\qquad \qquad
	\leq c_0 s^{\frac{1}{2}} \int_{-\infty }^\infty   (|x|^\beta+ s^{\beta/2}|x_0|^\beta)\Big( \Big( \frac{x^2}{s} \Big)^\eta + \Big( \frac{|x_0x|}{s^{1/2}} \Big)^{2\eta}\Big) e^{-\frac{x_0^2}{ 2}} dx_0 \\
&\qquad \qquad
	\leq c_0 (|x|^{\beta +2 \eta} s^{\frac{1}{2} - \eta}  + |x|^{2\eta } s^{\frac{1}{2} + \frac{\beta}{2} - \eta}).
\end{align*}
Setting this into (\ref{eq:Fixed_RequiredIntergal_I}) and making the change of variables $x \mapsto u^{1/2}x$ we obtain
\begin{align*}
\mathbb{E}[W_s^{q\zeta} p_u(W_s)^q \mathbf{1}_{s < \tau^W}]
	& \leq  c_0 (u^{\frac{\beta}{2} + \eta + \frac{q\zeta}{2} -\frac{q-1}{2}}s^{-\eta} + u^{\eta + \frac{q\zeta}{2} -\frac{q-1}{2}} s^{\frac{\beta}{2} - \eta}).
\end{align*}

Finally, the result for $X$ in place of $W$ follows by invoking Lemma \ref{Fixed_Lem_Girsanov}.  
\end{proof}

\begin{lem}[$I_2$]
\label{Fixed_Lem_I2}
If $\ell \in \mathcal{S}(\gamma, A, t_0)$, then there exist constants $c_1 > 0$, $\theta_1 > 0$ and $\theta_2 \in [0,1)$ (depending on $\ell$) such that 
\[
0 \leq I_2 \leq c_1 \delta^{1 + \theta_1} t^{-\frac{1}{2}\theta_2 },
	\qquad \textrm{for every } t \in (0,t_0).
\]
\end{lem}

\begin{rem}
Formally, the integrand in $I_2$ converges to a Dirac mass at zero, evaluated at $X_s$ and multiplied by $\mathbf{1}_{s < \tau}$. We expect such an expression to vanish since $s < \tau$ implies $X_s > 0$. While this argument is not rigorous, it suggests why $I_2 = o(\delta)$ and hence why this term will not contribute to the bound on $\Gamma[\ell]'$ in the proof of Theorem \ref{Intro_Thm_FixedPoint}. 
\end{rem}

\begin{proof}[Proof of Lemma \ref{Fixed_Lem_I2}]
We begin by applying the bound $\ell'_s \leq A s^{-\gamma}$ to get
\begin{align*}
\label{eq:Fixed_I2proof_I}
I_2 = 2\alpha\int^{t+\delta}_t \mathbb{E}[p_{t + \delta - s}(X_s) \mathbf{1}_{s<\tau}] d\ell_s
	&=2\alpha \int^{t+\delta}_t \mathbb{E}[p_{t + \delta - s}(X_s) \mathbf{1}_{s<\tau}] \ell_s' ds \nonumber \\
	& \leq 2\alpha A t^{-\gamma} \int^{t + \delta}_t \mathbb{E}[p_{t + \delta - s}(X_s) \mathbf{1}_{s<\tau}] ds.
\end{align*}
Taking $u = t + \delta - s$ and $\zeta = 0$ in Corollary \ref{Fixed_Cor_MainIntegralCaculation}, and bounding with the worst-case exponents, gives 
\[
I_2 \leq c_2 t^{-\gamma - \frac{\eta}{q}} \int^{t+\delta}_t (t+ \delta -s)^{\frac{\eta}{q} - \frac{1}{2}(1 - \frac{1}{q}) } ds
	= c_3 \delta^{1 + \frac{\eta}{q} - \frac{1}{2}(1 - \frac{1}{q})} t^{-\gamma - \frac{\eta}{q}},
\]
where $c_2$ and $c_3$ are constants depending on $\ell$. The result is now complete by choosing $\eta$ and $q$ so that $\gamma + \tfrac{\eta}{q} < \tfrac{1}{2}$ (possible since $\gamma < \tfrac{1}{2}$) and $\frac{\eta}{q} - \frac{1}{2}(1 - \frac{1}{q}) > 0$ (take $q$ sufficiently close to 1 and maintain constant ratio $\eta/ q$).
\end{proof}

The third term, $I_3$, is not so simple to control.

\begin{lem}[$I_3$]
\label{Fixed_Lem_I3}
If $\gamma < \tfrac{1}{2}$ and $\ell \in \mathcal{S}(\gamma, A, t_0)$ then there exist constants $c_0 >0$, $c_1 > 0$ and $\theta_3 > 0$ (depending only on $V_0$ and $\gamma$) such that
\[
|I_3| \leq c_0 \alpha A \delta \exp\{ c_1 \alpha^2 A^2 t_0^{1 - 2\gamma} \} t^{-\frac{1-\beta}{2} + \theta_3},
	\qquad \textrm{for every } t \in (0,t_0).
\]
\end{lem}

\begin{proof}
We begin by applying the fundamental theorem of calculus to see that
\begin{align*}
p_{t+\delta-s}(X_s) - p_{t-s}(X_s)
	&= \int^{\delta}_{0} \partial_u p_{t-s+u}(X_s) du \\
	&= \tfrac{1}{2} \int^{\delta}_{0} ( X_s^2 (t-s+u)^{-2} - (t-s+u)^{-1}) p_{t-s+u}(X_s) du
\end{align*}
and therefore
\[
|p_{t+\delta-s}(X_s) - p_{t-s}(X_s)|
	\leq  \frac{1}{2}  \int^{\delta}_0 ( X_s^2 (t-s)^{-2} + (t-s)^{-1}) p_{t-s+u}(X_s) du.
\]
Taking an expectation and using Corollary \ref{Fixed_Cor_MainIntegralCaculation} gives
\begin{align}
\mathbb{E}|&p_{t+\delta-s}(X_s) - p_{t-s}(X_s)|  \nonumber \\
	&\leq c_0 \delta \exp\Big\{ \frac{\alpha^2 \Vert \ell' \Vert^2_{L^2(0,t_0)}}{2(q-1)^2}  \Big\} \Big( 
		(t-s)^{\frac{\beta}{2} + \eta - \frac{q-1}{2} - 1} s^{-\eta} + (t-s)^{\eta - \frac{q-1}{2}- 1} s^{\frac{\beta}{2} - \eta} \Big)^{\frac{1}{q}}, \label{eq:epbd}
\end{align}
where $c_0$ depends only on $V_0$ and the choice of $q$ and $\eta$. 

By applying the bound \eqref{eq:epbd} to the expression for $I_3$, we get
\begin{multline*}
|I_3|
	\leq 2 A c_0 \alpha \delta \exp\Big\{ \frac{\alpha^2 \Vert \ell' \Vert^2_{L^2(0,t_0)}}{2(q-1)^2}  \Big\} \Big(J( \tfrac{1}{q}(\tfrac{\beta}{2} + \eta - \tfrac{q-1}{2} - 1), -\tfrac{\eta}{q} - \gamma )\\
	+ 
	J(\tfrac{1}{q}(\eta - \tfrac{q-1}{2} - 1), \tfrac{\beta}{2}-\tfrac{\eta}{q} - \gamma )
	\Big),
\end{multline*}
where $J$ is defined to be
\begin{equation}
\label{eq:Fixed_DefOfJIntegral}
J(a,b) := \int^t_0 (t-s)^a s^b ds
	= C t^{1 + a + b},
\end{equation}
for a constant $C = C(a,b) > 0$, provided $a > -1$ and $b > -1$. To keep the above exponents bigger than $-1$, we need to select $\eta$ and $q$ so that $\frac{\eta}{q} < 1 - \gamma$,
whereby we obtain
\begin{align*}
|I_3|
	&\leq A c_0 \alpha \delta \exp\Big\{ \frac{\alpha^2 \Vert \ell' \Vert^2_{L^2(0,t_0)}}{2(q-1)^2}  \Big\} t^{\frac{\beta}{2q} - \frac{1}{2}(\frac{1}{q} - 1) - \gamma} \\
	&= A c_0 \alpha \delta \exp\Big\{ \frac{\alpha^2 \Vert \ell' \Vert^2_{L^2(0,t_0)}}{2(q-1)^2}  \Big\} t^{-\frac{1-\beta}{2} + \frac{1-\beta}{2}(1 - \frac{1}{q}) + \frac{1}{2} - \gamma},
\end{align*}
where we have absorbed numerical constants into $c_0$. Since $\gamma < \tfrac{1}{2}$, we can take $q$ sufficiently close to 1 so that we have the required exponent. The proof is then complete by noting that
\[
\Vert \ell' \Vert^2_{L^2(0,t_0)}
	\leq \int^{t_0}_0 A^2 s^{-2\gamma} ds = \frac{A^2}{1 - 2\gamma} t_0^{1 - 2 \gamma}. 
\] 
\end{proof}

\begin{prop}[Stability of $\Gamma$]
\label{Fixed_Prop_Stability}
There exists a constant $K > 0$ depending only on $C$ and $x_\star$ such that for every $\varepsilon > 0$ there exists $t_0 > 0$ for which
\[
\Gamma : \mathcal{S}(\tfrac{1-\beta}{2}, K + \varepsilon, t_0) \to \mathcal{S}(\tfrac{1-\beta}{2}, K + \varepsilon, t_0),
\]
where $t_0$ also depends on the model parameters. 
\end{prop}

\begin{proof}
Take $K$ and $t_0$ as in the conclusion of Lemma \ref{Fixed_Lem_BoundTheBits}. We will decrease the value of $t_0$ throughout the proof, but this is the $K$ in the statement of the result.

First we check that $\Gamma : \mathcal{S}(\tfrac{1-\beta}{2}, K + \varepsilon, t_0) \to H^1(0,t_0)$. Combine Lemmas \ref{Fixed_Lem_BoundTheBits}, \ref{Fixed_Lem_I2} and \ref{Fixed_Lem_I3} to get
\begin{equation}
\label{eq:Fixed_Prop_Stability_I}
\Big|\frac{\Gamma[\ell]_{t+\delta} - \Gamma[\ell]_{t}}{\delta}\Big|
	\leq Kt^{-\frac{1-\beta}{2}} + c_2 \delta^{\theta_1} t^{-\frac{1}{2} \theta_2} + c_0 \alpha (K+\varepsilon) \exp\{ c_1 \alpha^2 (K+\varepsilon)^2 t_0^{\beta} \} t^{-\frac{1-\beta}{2} + \theta_3},
\end{equation}
and so we have
\[
\limsup_{\delta \to 0}\int^{t_0}_0 \Big|\frac{\Gamma[\ell]_{t+\delta} - \Gamma[\ell]_{t}}{\delta}\Big|^2 dt < \infty.
\]
By the method of difference quotients \cite[Sect.~5.8.2, Thm.~3]{evans_2010} we conclude that $\Gamma[\ell]$ is in $H^1(0,t_0)$, as required. 

Taking a pointwise limit in (\ref{eq:Fixed_Prop_Stability_I}) gives
\begin{align*}
|\Gamma[\ell]_t'| 
	&\leq Kt^{-\frac{1-\beta}{2}}  + c_0 \alpha (K+\varepsilon) \exp\{ c_1 \alpha^2 (K+\varepsilon)^2 t_0^{\beta} \} t^{-\frac{1-\beta}{2} + \theta_3} \\
	&\leq Kt^{-\frac{1-\beta}{2}}  + c_0 \alpha (K+\varepsilon) \exp\{ c_1 \alpha^2 (K+\varepsilon)^2 t_0^{\beta} \}t_0^{\theta_3} t^{-\frac{1-\beta}{2}}.
\end{align*}
With $\varepsilon > 0$ fixed, we can now take $t_0 > 0$ sufficiently small (since the constants in the bound in Lemma \ref{Fixed_Lem_I3} do not depend on the value of $t_0$) so that 
\[
c_0 \alpha (K+\varepsilon) \exp\{ c_1 \alpha^2 (K+\varepsilon)^2 t_0^{\beta} \}t_0^{\theta_3} < \varepsilon,
\]
which completes the proof.
\end{proof}

With Proposition \ref{Fixed_Prop_Stability} now in place, it remains to show that we can find a fixed point and deduce that it lives in one of the sets $\mathcal{S}(\frac{1-\beta}{2}, K+\varepsilon, t_0)$.

\begin{proof}[Proof of Theorem \ref{Intro_Thm_FixedPoint}]
Take $K$, $\varepsilon$ and $t_0$ as in the conclusion of Proposition \ref{Fixed_Prop_Stability}. Take $t_0$ sufficiently small so that the first half of Theorem \ref{Intro_Thm_Minimality} holds. Define the sequence
\[
\ell^{(0)} := 0,
	\qquad \ell^{(n)} := \Gamma[\ell^{(n-1)}],
		\qquad \textrm{for every }n \geq 1. 
\] 
By the Banach fixed point theorem, we know that there exists a limit point $\ell^{(n)} \to L$ in $L^\infty(0,t_0)$, as $n \to \infty$, and that $\Gamma[L] = L$. So $L$ solves (\ref{eq:Intro_MVproblem}) on $[0,t_0)$.

To see $L \in \mathcal{S}(\tfrac{1-\beta}{2}, K+\varepsilon, t_0)$, notice that, since $\Gamma[\ell]$ is bounded by 1, dominated convergence gives
\begin{equation}
\label{eq:Fixed_MainProof_I}
\int^{t_0-\delta}_{0} \Big|  \frac{L_{t+\delta} - L_t}{\delta}\Big|^2 dt
	\leq \limsup_{n \to \infty} \int^{t_0-\delta}_0 \Big|  \frac{\ell^{(n)}_{t+\delta} - \ell^{(n)}_t}{\delta}\Big|^2 dt,
		\qquad \textrm{for every } \delta > 0.
\end{equation}
Proposition \ref{Fixed_Prop_Stability} ensures $\ell^{(n)} \in \mathcal{S}(\tfrac{1-\beta}{2}, K+\varepsilon, t_0)$, so we have the estimate
\[
|\ell^{(n)}_{t + \delta} - \ell^{(n)}_t|
	=\Big| \int^{t+\delta}_t (\ell^{(n)}_s)' ds \Big|
	\leq (K+\varepsilon) t^{-\frac{1-\beta}{2}} \delta. 
\]
Using this in (\ref{eq:Fixed_MainProof_I}) yields
\[
\int^{t_0-\delta}_{0} \Big|  \frac{L_{t+\delta} - L_t}{\delta}\Big|^2 dt
	\leq (K+\varepsilon) \int^{t_0}_0   t^{-\frac{1-\beta}{2}} dt < \infty,
\]
and hence the method of difference quotients \cite[Sect.~5.8.2, Thm.~3]{evans_2010} gives that $L$ is in $H^1(0,t_0)$. Moreover, it holds pointwise almost everywhere in $(0,t_0)$ that
\[
0 \leq \frac{L_{t+\delta} - L_t}{\delta} 
	=\lim_{n \to \infty} \frac{\ell^{(n)}_{t+\delta} - \ell^{(n)}_t}{\delta} 
	\leq (K+\varepsilon) t^{-\frac{1-\beta}{2}},
\]
so sending $\delta \to 0$ gives that $L \in \mathcal{S}(\frac{1-\beta}{2}, K+\varepsilon, t_0)$, as required.

The uniqueness statement follows immediately by applying the second half of Theorem \ref{Intro_Thm_Minimality}.
\end{proof}

\section{Bootstrapping and full uniqueness up to explosion time --- proof of Theorem \ref{Intro_Thm_MainUniqueness}}
\label{Sect_Bootstrap}

Here we show how the solutions from the fixed point argument in the previous section can be extended up to the first time their $H^1$ norm explodes. (Trivially, this time occurs before or at the first jump time.) The key to this bootstrapping method is to notice that if $\ell \in H^1(0,t_0)$, for some $t_0$, then necessarily $V_{t_0-}(x) = O(x^{\beta })$ (Lemma \ref{Boot_Lem_IC_recovery}). Not only does this show that $t_0$ cannot be a jump time, it also allows us to apply the fixed point argument once more, thus extending the solution to $H^1(0,t_1)$ for some $t_1 > t_0$ (Corollary \ref{Boot_Cor_ExtendSolutionOneStep}). The first half of Theorem \ref{Intro_Thm_MainUniqueness} then follows by iterating this argument (Corollary \ref{Boot_Cor_Bootstrap}). 

After the above, we proceed to prove the second half of Theorem \ref{Intro_Thm_MainUniqueness}. The idea is to consider modified initial conditions for which a fixed portion of the initial density is erased and added to the initial value of the loss process (Definition \ref{Boot_Def_DeletedSolution}). An argument that shows solutions cannot overlap (Lemma \ref{Boot_Lem_Monotonicity}) then allows us to trap any general c\`{a}dl\`{a}g solution of (\ref{eq:Intro_MVproblem}) between the modified solutions and the minimal differentiable solution. The proof concludes by showing that the size of this trapping envelope shrinks to zero as the size of the initial modification is taken to zero (Lemma \ref{Boot_Lem_Convergence}), thus forcing the minimal solution and the general c\`{a}dl\`{a}g solution to be equal (see Figure \ref{fig:Boot_trapping}).

\subsection{Bootstrap}
\label{SubSect_Bootstrap}

\begin{lem}[Recovery of initial exponents]
\label{Boot_Lem_IC_recovery}
Suppose that $L \in H^1(0,t_0)$ solves (\ref{eq:Intro_MVproblem}) for some $t_0 > 0$, with the usual assumption (\ref{eq:Intro_DensityAssumption}) on $V_0$. Then there exists further constants $C', D', x_\star'  > 0$ such that
\[
V_{t_0-}(x) \leq C' x^{\beta} \mathbf{1}_{ x < x_\star' } + D' \mathbf{1}_{ x \geq x_\star' },
	\qquad \textrm{for every } x > 0. 
\]
\end{lem}

\begin{proof}
Note that there exists a limit $L_{t_0 -} = \lim_{s \uparrow t_0} L_s$, since $L$ must be increasing. Therefore, there exists a left limit density $V_{t_0-}$ (recall Proposition \ref{MinimalJumps_Prop_DensityExists}).

Fix $t < t_0$. Our strategy is to apply Hunt's switching identity \cite[Thm.~II.1.5]{bertoin_1996}. Define the dual process
\[
d\widehat{X}_s = dB_s + \alpha d\widehat{L}_s,
	\qquad \widehat{L}_s = L_{t-s},
	\qquad \widehat{\tau} = \inf\{ s > 0:  \widehat{X}_s \leq 0 \},
\]
for $s \in [0,t]$. Then we have 
\[
 \int_{x \in \mathbb{R}} \phi(x) \mathbb{E}_{X_0 = x} [ \psi(X_t) \mathbf{1}_{t < \tau} ] dx
 	= \int_{x \in \mathbb{R}} \psi(x) \mathbb{E}_{\widehat{X}_0 = x} [ \phi(\widehat{X}_t) \mathbf{1}_{t < \widehat{\tau}} ] dx,
\]
for all non-negative measurable functions $\phi$ and  $\psi$. (Note that, since $L \in H^1(0,t_0)$ and $\Delta L_{t_0} = 0$ by the hypotheses, we have continuity of $L$).  By taking $\phi = V_0$ and $\psi$ an arbitrary non-negative measurable function, we can conclude that
\begin{equation}
\label{eq:Boot_IC_recovery_I}
V_t(x) = \mathbb{E}_{X_0 = x} [V_0(\widehat{X}_t) \mathbf{1}_{t < \widehat{\tau}}]
\end{equation}
for almost every $x > 0$.

	Since $L \in H^1(0,t_0)$, $\widehat{X}$ is just a Brownian motion with a $H^1$ drift. Therefore we can apply Lemma \ref{Fixed_Lem_Girsanov} with $\widehat{X}$ in place of $X$ and $F = V_0$ to get 
\[
\mathbb{E}_{X_0 = x} [V_0(\widehat{X}_t) \mathbf{1}_{t < \widehat{\tau}}]
	\leq C_q \mathbb{E}_{W_0 = x} [(V_0(W_t) \mathbf{1}_{t < \tau^W})^q]^{\frac{1}{q}}
\]
	for any $q > 1$, where $W$ is a Brownian motion and $C_q < \infty$ depends on $L$, $t_0$ and $q$ and is independent of $t$. (We have absorbed the exponential factor into $C_q$.) By the assumption that $V_0(x) = O(x^\beta)$, and increasing $C_q$ as needed, we deduce that
\[
\mathbb{E}_{X_0 = x} [V_0(\widehat{X}_t) \mathbf{1}_{t < \widehat{\tau}}]
	\leq C_q \mathbb{E}_{W_0 = x} [ ( W_t \mathbf{1}_{t < \tau^W})^{q \beta}]^{\frac{1}{q}}.
\]
Provided we take $q>1$ such that $q\beta < 1$, we can apply Jensen's inequality to get
\[
\mathbb{E}_{X_0 = x} [V_0(\widehat{X}_t) \mathbf{1}_{t < \widehat{\tau}}]
	\leq C_q \mathbb{E}_{W_0 = x} [W_t \mathbf{1}_{t < \tau^W}]^{\beta}
	= C_q x^{\beta}.
\]
Putting this into (\ref{eq:Boot_IC_recovery_I}) gives the required bound near zero at time $t$. Proposition \ref{MinimalJumps_Prop_DensityExists} gives $V_t(x) \leq \Vert V_0 \Vert_\infty$, so we have the required boundedness away from the origin too. Since the constants obtained above are independent of $t$, we can find $C'$, $D'$ and $x_\star'$ such that
\[
V_t(x) \leq C' x^{\beta} \mathbf{1}_{ x < x_\star' } + D' \mathbf{1}_{ x \geq x_\star' }
\] 
for all $x > 0$ and $t < t_0$. By sending $t \uparrow t_0$, we have the result. 
\end{proof}

The implication of Lemma \ref{Boot_Lem_IC_recovery} is that at the end of the fixed-point argument from Section \ref{Sect_FixedPoint} we can restart the argument with new initial conditions that have the same power law decay. As a result we can push our construction of solutions by a further non-zero amount of time. Notice, however, that we lose control of the exact constant that we had in Theorem \ref{Intro_Thm_FixedPoint} and hence the proceeding results are qualitative, not quantitative.

\begin{cor}[Extending solutions]
\label{Boot_Cor_ExtendSolutionOneStep}
Suppose we have a solution $L \in H^1(0,t_0)$ to (\ref{eq:Intro_MVproblem}). Then there exists $t_1 > t_0$ such that $L$ can be extended to a solution of (\ref{eq:Intro_MVproblem}) on $(0,t_1)$ for which $L \in H^1(0,t_1)$. Furthermore, if $L \in \mathcal{S}(\frac{1-\beta}{2}, K_0, t_0)$ for some $K_0 > 0$, then we can find $K_1 > 0$ such that the extension satisfies $L \in \mathcal{S}(\frac{1-\beta}{2}, K_1, t_1)$.
\end{cor}

\begin{proof}
Since we have $L \in H^1(0,t_0)$, Lemma \ref{Boot_Lem_IC_recovery} implies that $V_{t_0-}$ satisfies the condition (\ref{eq:Intro_DensityAssumption}), for some constants $C$, $D$ and $x_\star$ (that are possibly different to those for $V_0$). Therefore, by Theorem \ref{Intro_Thm_FixedPoint} there exists $K_1>0$ and $u_0 > 0$ such that we can find $F \in \mathcal{S}(\frac{1-\beta}{2}, K_1, u_0)$ solving
\begin{align}
\label{eq:Boot_MVproblem_restarted}
\begin{cases}
 X_{t_0 + u} = X_{t_0} + B_u - \alpha F_u \\
\tau^{(t_0)} = \inf\{ u \geq 0 : X_{t_0 + u} \leq 0 \}  \\
F_u = \mathbb{P}(\tau^{(t_0)} \leq u), 
\end{cases}
\end{align}
for $u < u_0$ (recall Remark \ref{MinimalJumps_Rem_StartingNotZero}). It follows that, if we define $t_1 := t_0 + u_0$ and 
\[
\widetilde{L}_{t}:=\begin{cases}
L_{t}, & \textrm{if }0<t<t_{0}\\
F_{t-t_{0}}, & \textrm{if }t_{0}\leq t<t_{1},
\end{cases}
\]
then $\widetilde{L}$ extends $L$ and solves (\ref{eq:Intro_MVproblem}) for all $t \in (0,t_1)$ with the required derivative control. 
\end{proof}

Naturally, we can iterate Corollary \ref{Boot_Cor_ExtendSolutionOneStep} to prove the first half of Theorem \ref{Intro_Thm_MainUniqueness}.

\begin{cor}[Bootstrap to explosion time]
\label{Boot_Cor_Bootstrap}
There exists a solution $L$ to (\ref{eq:Intro_MVproblem}) up to time
\[
	t_{\mathrm{explode}} := \sup\{t > 0 : \Vert L \Vert_{H^1(0,t)}  < \infty \} \in (0,\infty]
\]
such that, for every $t_0 < t_{\mathrm{explode}}$, we have $L \in \mathcal{S}(\frac{1-\beta}{2}, K, t_0)$ for some $K > 0$. 
\end{cor}

\begin{proof}
By repeating Corollary \ref{Boot_Cor_ExtendSolutionOneStep}, we can find an infinite sequence of times $t_0  < \cdots < t_n < \cdots$ over which we can successively extend $L$. If $t_\infty := \lim_{n \to \infty} t_n$ is such that $\Vert L \Vert_{H^1(0,t_\infty)} = \infty$, then we are done. Otherwise, by the left continuity of $L$, we can restart the argument from $t_\infty$ by applying Corollary \ref{Boot_Cor_ExtendSolutionOneStep}. This procedure cannot terminate at a time for which $\Vert L \Vert_{H^1(0,t)} < \infty$, or else it can be restarted, hence we conclude the result. 
\end{proof}

\subsection{Monotonicity and trapping} \label{Subsec:mon_trap}

Our main technical construction will be a solution to (\ref{eq:Intro_MVproblem}) for which we delete a portion of the initial condition near the boundary and add that mass to the loss at time zero, before finally shifting the density towards zero accordingly. 

\begin{defn}[$\varepsilon$-deleted solutions]
\label{Boot_Def_DeletedSolution}
Suppose $X_0$ has a density $V_0$ and let $\varepsilon > 0$. Define the McKean--Vlasov problem 
\begin{align}
\label{eq:Boot_DeletedSol}
\begin{cases}
 X_t^\varepsilon = X_0 \mathbf{1}_{X_0 \geq \varepsilon} -\tfrac{1}{4} \varepsilon + B_t - \alpha L_t^\varepsilon \\
\tau^\varepsilon = \inf\{ t \geq 0 : X^\varepsilon_t \leq 0 \}  \\
L_t^\varepsilon = \nu_0(0,\varepsilon) + \int_{\varepsilon}^{\infty}\mathbb{P}_{X_0 = x}(\tau^\varepsilon \leq t ) \nu_0(dx), 
\end{cases}
\end{align}
where solutions, $L^\varepsilon$, are taken to be c\`{a}dl\`{a}g. 
\end{defn}

It is not immediately clear that we can solve the above problem for $\varepsilon > 0$, however, careful inspection of the initial loss reveals that we can.

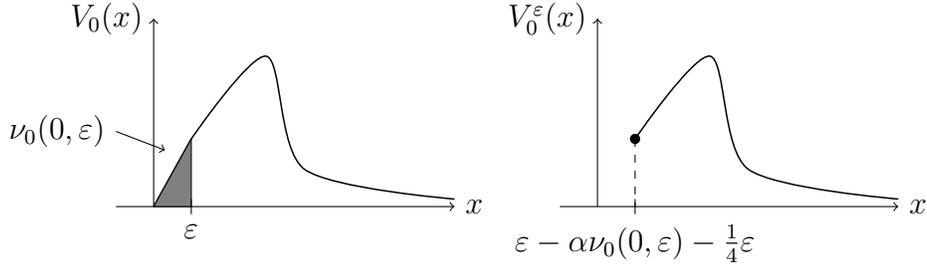
\begin{figure}
\begin{center}
\begin{tikzpicture}
\draw[->] (-0.5,0) -- (4,0) node [right] {$x$};
\draw[->] (0,0) -- (0,2.5) node [left] {$V_0(x)$};
\draw[-, line width = 0.2mm] plot coordinates { (0,0) (0.5, 0.9) } ;
\draw[-,  line width = 0.2mm] plot [smooth] coordinates { (0.5, 0.9)  (1.5,2) (2,0.5) (4, 0.1) } ;
\draw[fill = gray] plot coordinates { (0,0) (0.5, 0.9) (0.5,0) (0,0) } ;
\draw[-] (0.5,0.1) -- (0.5,-0.1) node [below] {$\varepsilon$};
\draw[<-] (0.15, 0.75) -- (-0.5, 1) node [left] {$\nu_0(0,\varepsilon)$};

\draw[->] (-0.5+6-0.1,0) -- (4+6-0.1,0) node [right] {$x$};
\draw[->] (0+6-0.1,0) -- (0+6-0.1,2.5) node [left] {$V^\varepsilon_0(x)$};
\draw[-,  line width = 0.2mm] plot [smooth] coordinates { (0.5 +6-0.1 , 0.9)  (1.5+6-0.1,2) (2+6-0.1,0.5) (4+6-0.1, 0.1) } ;
\node[circle, fill = black, scale = 0.35] at (0.5 +6-0.1 , 0.9) {};
\draw[-, dashed] (0.5 +6-0.1 , 0.9) -- (0.5 +6-0.1 , 0);
\draw[-] (0.5 +6-0.1,0.1) -- (0.5 +6-0.1,-0.1) node [below] {$\varepsilon - \alpha \nu_0(0,\varepsilon) - \tfrac{1}{4} \varepsilon$};

\end{tikzpicture}
\caption{\label{fig:Boot_eps_deleted_ic} Given an initial density, $V_0$, the $\varepsilon$-deleted initial condition constructed in (\ref{eq:Boot_Contruction_F_solution}) is obtained by killing the mass on $(0,\varepsilon)$ and shifting the density towards the origin by the amount $\alpha\nu_0(0,\varepsilon) + \tfrac{1}{4} \varepsilon$. The proof of Lemma \ref{Boot_Lem_epsilonSolution} shows us that $V^\varepsilon_0$ vanishes in a neighbourhood of zero. }
\end{center} \vspace{-10pt}
\end{figure} 

\begin{lem}[$\varepsilon$-deleted solutions exist]
\label{Boot_Lem_epsilonSolution}
Assume $\nu_0$ has a density satisfying (\ref{eq:Intro_DensityAssumption}). Then there exist $K > 0$, $t_0 > 0$ and $\varepsilon_0 > 0$ such that, for every $\varepsilon \in (0, \varepsilon_0)$, there is a solution $L^\varepsilon \in \mathcal{S}(\frac{1-\beta}{2}, K, t_0)$ to (\ref{eq:Boot_DeletedSol}). 
\end{lem}

\begin{proof}
We begin by noting that, for $\varepsilon < \varepsilon_0 < x_\star$, we have
\[
\nu_0(0,\varepsilon) \leq \int^\varepsilon_0 Cx^\beta dx \leq \frac{C}{1+\beta} \varepsilon^{1 + \beta}
	\leq \frac{C\varepsilon_0^\beta}{1+\beta} \varepsilon.
\]
Hence we can certainly take $\varepsilon_0$ sufficiently small so that 
\[
\nu_0(0,\varepsilon) \leq \tfrac{1}{4} \alpha^{-1} \varepsilon,
	\qquad \textrm{for every } \varepsilon < \varepsilon_0. 
\]
This guarantees that $\alpha \nu_0(0,\varepsilon) + \tfrac{1}{4}\varepsilon \leq \tfrac{1}{2}\varepsilon < \varepsilon$, and so we can rewrite (\ref{eq:Boot_DeletedSol}) as 
\begin{align}
\label{eq:Boot_Contruction_F_solution}
L^\varepsilon_t = \nu_0(0,\varepsilon) + F_t^\varepsilon,
	\qquad 
\begin{cases}
 X_t^{\varepsilon,x} = x + B_t - \alpha F_t^\varepsilon \\
\tau^{\varepsilon,x} = \inf\{ t \geq 0 : X^{\varepsilon,x}_t \leq 0 \}  \\
F_t^\varepsilon = \int_{0}^{\infty}\mathbb{P}(\tau^{\varepsilon,x} \leq t ) V_0^\varepsilon(x) dx, \\
V_0^\varepsilon(x) = V_0(x + \alpha \nu_0(0,\varepsilon) + \tfrac{1}{4}\varepsilon) \mathbf{1}_{x + \alpha \nu_0(0,\varepsilon ) + \varepsilon/4 \geq \varepsilon},
\end{cases}
\end{align}
see Figure \ref{fig:Boot_eps_deleted_ic}. As noted in Remark \ref{MinimalJumps_Rem_StartingNotZero}, we can solve to find $F^\varepsilon$, since $V_0^\varepsilon$ vanishes on $x < \tfrac{1}{2}\varepsilon$, so we certainly have the control in (\ref{eq:Intro_DensityAssumption}) for some choice of constants.   Furthermore, we have
\begin{align*}
V_0^\varepsilon(x)
	&\leq C (x + \alpha \nu_0(0,\varepsilon)+ \tfrac{1}{4}\varepsilon)^\beta \mathbf{1}_{x + \alpha \nu_0(0,\varepsilon) + \varepsilon /4  \geq \varepsilon}  \\
	&\leq C (x + \tfrac{1}{2}\varepsilon)^\beta  \mathbf{1}_{x  \geq \varepsilon / 2}
	\leq 2^\beta Cx^\beta, 
\end{align*}
so we can find constants such that (\ref{eq:Intro_DensityAssumption}) holds for $V^\varepsilon_0$ uniformly in $\varepsilon \in (0,\varepsilon_0)$. Since the arguments in Section \ref{Sect_FixedPoint} for proving Theorem \ref{Intro_Thm_FixedPoint} depend only on the constants in (\ref{eq:Intro_DensityAssumption}), we can conclude that the 
parameters of $\mathcal{S}$ obtained in Theorem \ref{Intro_Thm_FixedPoint} are constant across the class of initial densities $\{V^\varepsilon_0\}_\varepsilon$, for $\varepsilon < \varepsilon_0$. Hence we have the result. 
\end{proof}

The solutions in Lemma \ref{Boot_Lem_epsilonSolution} are useful because we have the following two comparison results, which say that the $\varepsilon$-deleted solutions dominate and converge to the traditional solutions of (\ref{eq:Intro_MVproblem}) --- see Figure \ref{fig:Boot_trapping}.

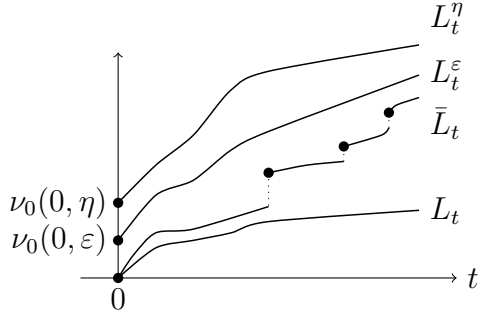
\begin{figure}
\begin{center}
\begin{tikzpicture}
\draw[->] (-0.5,0) -- (4.5,0) node [right] {$t$};
\draw[->] (0,0) -- (0,3) node [left] { };
\draw[-,  line width = 0.2mm] plot [smooth] coordinates { (0,0) (0.5, 0.4) (1.0,0.5) (1.5,0.6) (2,0.75) (4, 0.9) } ;
\node[circle, fill = black, scale = 0.35] at (0.0,0.0) {};

\draw[-,  line width = 0.2mm] plot [smooth] coordinates { (0,0.5) (0.5, 0.4+0.7) (1.0,0.5+0.8) (1.5,0.6+1.1) (2,0.75+1.2) (4, 0.9+1.5+0.3) } ;
\node[circle, fill = black, scale = 0.35] at (0,0.5) {};

\draw[-,  line width = 0.2mm] plot [smooth] coordinates { (0,0.5+0.5) (0.5, 0.4+0.7+0.4) (1.0,0.5+0.8+0.6) (1.5,0.6+1.1+0.8) (2,0.75+1.2+0.8) (4, 0.9+1.5+0.7) } ;
\node[circle, fill = black, scale = 0.35] at (0,1) {};

\draw[-,  line width = 0.2mm] plot [smooth] coordinates { (0,0) (0.2, 0.3) (0.5, 0.6) (1.0, 0.65) (2.0, 0.95) } ;
\node[circle, fill = black, scale = 0.35] at (2.0 , 1.4) {};
\draw[-,dotted] (2.0 , 1.4) -- (2.0,0.95);
\draw[-,  line width = 0.2mm] plot [smooth] coordinates { (2.0 , 1.4) (2.5, 1.5) (3,1.55)  } ;
\node[circle, fill = black, scale = 0.35] at (3,1.75) {};
\draw[-,dotted] (3,1.55) -- (3,1.75);
\draw[-,  line width = 0.2mm] plot [smooth] coordinates { (3,1.75) (3.5, 1.9)  (3.6, 2.0)  } ;
\node[circle, fill = black, scale = 0.35] at (3.6,2.2) {};
\draw[-,dotted] (3.6, 2.0)  -- (3.6,2.2);
\draw[-,  line width = 0.2mm] plot [smooth] coordinates { (3.6,2.2) (3.7, 2.3)  (4, 2.4)  };

\node[below right] at (4, 2.4) {$\bar{L}_t$};
\node[right] at (4, 0.9+1.5+0.3) {$L^\varepsilon_t$};
\node[above right] at (4, 0.9+1.5+0.7)  {$L^\eta_t$};
\node[right] at (4, 0.9) {$L_t$};

\node[left] at (0,0.5) {$\nu_0(0,\varepsilon)$};
\node[left] at (0,1.0) {$\nu_0(0,\eta)$};
\node[below] at (0,0) {$0$};

\end{tikzpicture}
\caption{\label{fig:Boot_trapping} On a small time interval, the unique differentiable solution, $L$,  from Corollary \ref{Boot_Cor_Bootstrap} and the $\varepsilon$-deleted solutions from Definition \ref{Boot_Def_DeletedSolution} trap any candidate c\`{a}dl\`{a}g solution, $\bar{L}$, by Lemma \ref{Boot_Lem_Monotonicity}. Here $\eta > \varepsilon$. Lemma \ref{Boot_Lem_Convergence} shows that $L^\varepsilon \to L$ uniformly on a small time interval as $\varepsilon \to 0$, and so we see that $L = \bar{L}$ is forced since the area between the curves above shrinks to zero.  }
\end{center} \vspace{-10pt}
\end{figure}

\begin{lem}[Monotonicity]
\label{Boot_Lem_Monotonicity}
Let $\varepsilon > 0$ and $t_0 > 0$ be fixed. Assume that $L$ is a generic solution to (\ref{eq:Intro_MVproblem}) and that $L^\varepsilon$ solves (\ref{eq:Boot_DeletedSol}) and that both are continuous on $[0,t_0)$. Then 
\[
L^\varepsilon_t > L_t,
	\qquad \textrm{for every } t \in [0,t_0).
\]
\end{lem}

\begin{proof}
By construction, the result is true at time $t= 0$. For a contradiction, let $t$ be the first time at which $L^\varepsilon_t = L_t$. Couple both solutions to the same Brownian motion:
\begin{align*}
X_s &= X_0  + B_s - \alpha L_s \\
X^\varepsilon_s &= X_0 \mathbf{1}_{X_0 \geq \varepsilon} -\tfrac{1}{4}\varepsilon + B_s   - \alpha L^\varepsilon_s,
\end{align*}
for $s \in [0,t)$, so that taking the difference gives
\[
X_s - X^\varepsilon_s
	= X_0 \mathbf{1}_{X_0 <  \varepsilon} + \tfrac{1}{4}\varepsilon - \alpha(L_s - L^\varepsilon_s)
	\geq \tfrac{1}{4}\varepsilon.
\]
Therefore, from the definition of $L^\varepsilon_t$, we get
\begin{align}
\label{eq:Boot_ComparisonI}
L^\varepsilon_t 
	= \mathbb{P}(\inf_{0\leq u \leq t} X^\varepsilon_u \leq 0)
	&= \lim_{s \uparrow t}  \mathbb{P}(\inf_{0\leq u \leq s} X^\varepsilon_u \leq 0) \nonumber \\
	&\geq \lim_{s \uparrow t} \mathbb{P}(\inf_{0\leq u \leq s} X_u \leq \tfrac{1}{4}\varepsilon) \nonumber  \\
	&= L_t + \mathbb{P}( \inf_{0\leq u \leq t} X_u \in (0, \tfrac{1}{4}\varepsilon])
\end{align}
Now fix any $f \in H^1([0,t])$ such that $\Vert \alpha L - f \Vert_{L^\infty([0,t])} < \tfrac{1}{16}\varepsilon$. Then we have
\begin{equation}
\label{eq:Boot_ComparisonII}
L^\varepsilon_t \geq L_t + \mathbb{P}(\inf_{0\leq u \leq t} \{ X_0 + B_u - f_u \} \in (\tfrac{1}{16}\varepsilon, \tfrac{3}{16} \varepsilon)).
\end{equation}
By Girsanov's Theorem, $u \mapsto B_u - f_u$ is absolutely continuous with respect to Brownian motion, and the infimum of Brownian motion has a density, therefore the probability on the right-hand side of (\ref{eq:Boot_ComparisonII}) is non-zero. Hence $L^\varepsilon_t > L_t$, which is the required contradiction. 
\end{proof}

\begin{lem}[Convergence]
\label{Boot_Lem_Convergence}
Suppose there exist constants $\gamma \in (0,\tfrac{1}{2})$, $K >0$, $t_0 > 0$ and $\varepsilon_0 > 0$ such that $L, L^\varepsilon \in \mathcal{S}(\gamma, K, t_0)$ for all $\varepsilon \in (0,\varepsilon_0)$. Then there exists $t_1 \in (0 , t_0]$ such that
\[
\Vert L^\varepsilon - L \Vert_{L^\infty(0,t_1)} \to 0,
	\qquad \textrm{as } \varepsilon \to 0.
\]
\end{lem}

\begin{proof}
We know from Lemma \ref{Boot_Lem_Monotonicity} that $L^\varepsilon > L$ on $[0,t_0)$. By following the argument in Section \ref{Sect_Unique} for the proof of Theorem \ref{Intro_Thm_Minimality}, and coupling $L$ and $L^\varepsilon$ to the same Brownian motion, we have
\begin{align*}
0\leq L^\varepsilon_t - L_t 
	= \mathbb{P}(\tau^\varepsilon \leq t < \tau)
	&= \int^t_0 \mathbb{P}(t < \tau | \tau^\varepsilon = s) dL^\varepsilon_s + \nu_0(0,\varepsilon) \\
	&\leq \int^t_0 \Phi\Big( \frac{\alpha(L^\varepsilon_s - L_s) + \tfrac{1}{4}\varepsilon }{(t-s)^{1/2}} \Big) (L^\varepsilon)'_s ds + \nu_0(0,\varepsilon) \\
	&\leq c_0(\Vert L^\varepsilon - L \Vert_{L^\infty(0,t)} + \tfrac{1}{4} \varepsilon) \int^t_0 \frac{ds}{(t-s)^{1/2}s^\gamma} + \nu_0(0,\varepsilon) \\
	&\leq c_0 t^{1/2 - \gamma}(\Vert L^\varepsilon - L \Vert_{L^\infty(0,t)} + \tfrac{1}{4} \varepsilon)+ \nu_0(0,\varepsilon), 
\end{align*}
where $c_0 > 0$ is a constant independent of $t$, increasing from line to line as necessary. Therefore, taking a supremum over $t \leq t_1$ and taking $t_1$ sufficiently small, we have
\[
\Vert L^\varepsilon - L \Vert_{L^\infty(0,t_1)}
	\leq \tfrac{1}{2} \Vert L^\varepsilon - L \Vert_{L^\infty(0,t_1)} + o(1),
\]
which completes the proof. 
\end{proof}

We are now in a position to complete the proof of the main uniqueness theorem. As already indicated, the idea is to trap a general candidate solution from below by the differentiable solution from Corollary \ref{Boot_Cor_Bootstrap} and from above by the $\varepsilon$-deleted solutions, and then to shrink the resulting envelope to zero using the uniform control in Lemma \ref{Boot_Lem_Convergence}

\begin{proof}[Proof of Theorem \ref{Intro_Thm_MainUniqueness}]
Take the bootstrapped solution, $L$, from Corollary \ref{Boot_Cor_Bootstrap}. We will be done if we can show that there is not a solution to (\ref{eq:Intro_MVproblem}) on $[0,t_\star)$ that is distinct from $L$, for any $t_\star \leq t_{\mathrm{explode}}$. So for a contradiction, suppose that $\bar{L}$ is such a solution and let $t_{\mathrm{jump}}:=\inf\{t>0 : \Delta \bar{L}_t >0 \}$ be the first time that discontinuities of $t \mapsto \bar{L}_t$ emerge. 

By Lemma \ref{Boot_Lem_epsilonSolution} we can find $K>0$, $t_0 > 0$ and a family $L^\varepsilon$ satisfying the hypotheses of Lemma \ref{Boot_Lem_Convergence}. By decreasing $t_0$ so that $t_0 < t_{\mathrm{jump}}$, Lemma \ref{Boot_Lem_Monotonicity} and Theorem \ref{Intro_Thm_Minimality} guarantee
\[
L_t \leq \bar{L}_t < L^\varepsilon_t,
	\qquad \textrm{for every } t < t_0.
\]
By taking $t_1 \in [0, t_0)$ from the conclusion of Lemma \ref{Boot_Lem_Convergence} and sending $\varepsilon \to 0$, we conclude that $L = \bar{L}$ on $[0,t_1]$.

This argument can now be restarted from time $t_1$ and iterated as in the proof of Corollary \ref{Boot_Cor_Bootstrap}, hence we conclude that $L = \bar{L}$ up to the minimum of $t_\mathrm{jump}$ and $t_\star$. If $t_\star \leq t_\mathrm{jump}$, then we are done, so suppose $t_\mathrm{jump}<t_\star$. By left-continuity, we have that $\nu_{t_\mathrm{jump}-} = \bar{\nu}_{t_\mathrm{jump}-}$, and, since $L$ does not have a jump at time $t_\mathrm{jump}$, the physical jump condition (\ref{eq:Intro_JumpCondition}) gives
\[
\Delta\bar{L}_{t_\mathrm{jump}-} =\inf\{ x \geq 0 : \bar{\nu}_{t_\mathrm{jump}-}(0,\alpha x) < x \}  = \inf\{ x \geq 0 : \nu_{t_\mathrm{jump}-}(0,\alpha x) < x \} = 0.
\]
Now, $ \bar{\nu}_{t_\mathrm{jump}-}$ must also satisfy \eqref{eq:Intro_DensityAssumption}, so, by the right-continuity of $\bar{L}$, we can deduce that the physical jump condition gives $\Delta\bar{L}_{t} =0$ for all $t$ in a (nonempty) right neighbourhood of $t_{\mathrm{jump}}$ (for details, see \cite[Proposition 6.4.3]{sojmark_2019}). This contradicts the definition of $t_{\mathrm{jump}}$ and so it completes the proof. 
\end{proof}

\section{Extensions to more general coefficients}
\label{Sect_GeneralCoeff}
In this section we consider some simple extensions of (\ref{eq:Intro_MVproblem}) that incorporate more general drift and diffusion coefficients. The aim is to outline how the analysis can be reduced to that of the previous sections and thus we will only provide sketches of the proofs.

To be specific, we consider the McKean\textendash Vlasov problem
\begin{equation}
\label{eq:General McKean}
\begin{cases}
dX_{t}=b(t,X_{t})dt+\sigma(t)dB_{t}-\alpha dL_{t}\\
\tau=\inf\left\{ t\geq0:X_{t}\leq0\right\} \\
L_{t}=\mathbb{P}(\tau\leq t),
\end{cases}
\end{equation}
where $B$ is a Brownian motion and the independent initial condition,
$X_{0}$, is given by $\nu_{0}$, which is assumed to satisfy (\ref{eq:Intro_DensityAssumption}) and taken to be sub-Gaussian. In terms of the coefficients, we assume that $b(t,x)$ is Lipschitz in space with $\left|b(t,x)\right|\leq C(1+\left|x\right|)$ and we impose the non-degeneracy condition $c\leq\sigma(t)\leq C$ for all $t\in[0,T]$. Furthermore, we impose an upper bound $b(t,x) \leq M$ for positive $x \geq 0$ (this is only used in Section \ref{Sec63} and it could be omitted if we could ensure $\kappa=1$ in Section \ref{Sec61}).  

The above setup includes the financial model (\ref{eq:Intro_FiniteSystem_Def}). Moreover, it e.g.~allows for a Brownian motion with an Ornstein--Uhlenbeck type drift modelling the attraction to a `resting state' --- as in the original neuroscience model \cite{dirt_annalsAP_2015}. This could also be of interest in the financial framework,  for example to include a target leverage ratio or to model some notion of `flocking to default' as in  \cite{fouque_sun_2013}.

\begin{thm}[Existence and uniqueness up to explosion]
	\label{Thm_GenUniqueness}
	Let the above assumptions be in place. Then there exists a solution to (\ref{eq:General McKean}) up to time
	\[
	t_\mathrm{explode} 
	:= \sup\{
	t > 0: \Vert L \Vert_{H^1(0,t)} < \infty
	\} \in (0,\infty]
	\]
	such that, for some $\bar{\beta}>0$, it holds for every $t_0 < t_\mathrm{explode} $ that $L \in \mathcal{S}(\tfrac{1-\bar{\beta}}{2}, K, t_0)$ for some $K > 0$. Furthermore, this solution is unique in the class of candidate solutions with $L$ in
	$\bigcup_{0<\gamma<1/2, \thinspace A >0} \mathcal{S}(\gamma, A, t_0)$.  
\end{thm}

As for (\ref{eq:Intro_MVproblem}), the notions of weak and strong uniqueness for (\ref{eq:General McKean}), regarded as an SDE in $X$, are equivalent: taking the difference of two solutions with the same $B$ and $L$ cancels all terms apart from the drift, whereby strong uniqueness is immediate from the Lipschitz property of $x \mapsto b(t,x)$. In essence, the main point in the proof of Theorem \ref{Thm_GenUniqueness} is to have good control on the boundary decay of the relevant transition densities akin to the classical Dirichlet heat kernel estimates.

When there is no spatial dependence in the drift --- as in (\ref{eq:Intro_FiniteSystem_Def}) --- we can do better with essentially no extra work. Specifically, we can directly replicate the monotonicity and trapping arguments from Section \ref{Subsec:mon_trap}, which gives uniqueness amongst generic c\`adl\`ag solutions.
\begin{thm} [Generic uniqueness]
	\label{Thm_Generic_uniqueness_b(t)}
	Suppose the above assumptions are in place and that, in addition, $b(t,x)=b(t)$.
Then we have uniqueness on $[0,t_\mathrm{explode})$ amongst all generic c\`adl\`ag solutions to (\ref{eq:General McKean}) in the same sense as Theorem \ref{Intro_Thm_MainUniqueness}.
\end{thm}

The remainder of this section concerns the proofs of Theorems \ref{Thm_GenUniqueness} and \ref{Thm_Generic_uniqueness_b(t)}, however, we will  only show how to reduce the analysis to that of Sections \ref{Sect_Unique}, \ref{Sect_FixedPoint} and \ref{Sect_Bootstrap}.
The proof of Theorem \ref{Thm_GenUniqueness} is split into three succinct parts: first we discuss the stability of the fixed point map (Section \ref{Sec61}), then we revisit the bootstrapping of short-time existence (Section \ref{Sec62}), and finally we extend the contractivity arguments (Section \ref{Sec63}).
Last, the proof of Theorem \ref{Thm_Generic_uniqueness_b(t)} is  outlined in Section \ref{Sec64}.

\subsection{Stability of the fixed point map}
\label{Sec61}
Let  $\Gamma$ denote the fixed point map for (\ref{Thm_GenUniqueness}) defined by analogy with (\ref{eq:Intro_DefnOfGamma}). Then we need to ensure that Proposition \ref{Fixed_Prop_Stability}, concerning the stability of $\Gamma$, also holds true in this more general setting.

To this end, the starting point is still the expression
for $\Gamma[\ell]$ from Lemma \ref{Fixed_Lem_StartingPoint}, which remains valid with $X$ given by (\ref{eq:General McKean}). Next, we can observe that the Girsanov argument from Lemma \ref{Fixed_Lem_Girsanov} will now take the form
\[
\mathbb{E}[F(X_t)\mathbf{1}_{t < \tau}] 
\leq \exp\Big\{ \frac{\alpha^2 \Vert \ell' \Vert^2_{L^2(0,t)} }{2 c^{2}(q-1)^2}   \Big\} \mathbb{E}[F(\widetilde{X}_t)^q\mathbf{1}_{t < \widetilde{\tau}}]^{\frac{1}{q}},  
\]
with
\[
d\widetilde{X}_{t}=b(t,\widetilde{X}_{t})dt+\sigma(t)dW_{t}\quad\text{and}\quad\widetilde{\tau}=\inf\bigl\{ t>0:\widetilde{X}_{t}\leq0\bigr\},
\]
where $W$ is a standard Brownian motion and $\widetilde{X}_{0}$ is an independent random variable distributed according to $\nu_0$. Indeed, we can simply consider the Radon--Nikodym derivative
\[
\frac{d\mathbb{Q}}{d\mathbb{P}}\Big|_{\mathcal{F}_t}= \exp\Big\{ 
\alpha \int^t_0 \frac{\ell_s'}{\sigma(s)} dB_s - \frac{\alpha^2}{2}  \int^t_0 \frac{(\ell_s')^2}{\sigma(s)^2} ds
\Big\}
\]
and otherwise repeat the arguments from the proof of Lemma \ref{Fixed_Lem_Girsanov}. From here, we can apply \cite[Thm.~2.7]{hambly_sojmark_2018} to see that the absorbed process $\widetilde{X}$  has a transition density $\widetilde{V}_{t}(x,x_0)$ which satisfies the Dirichlet heat kernel type estimate
\begin{equation}
\label{absorbing_density_est}
\widetilde{V}_{t}(x,x_0)\leq C\Bigl( \bigl(t^{-1}xx_0\land1\bigr)t^{-\frac{1}{2}}+\bigl(t^{-\kappa}x^{\kappa}x_0^{\kappa}\land1\bigl)e^{\delta x_{0}^{2}}\Bigr)e^{-c(x-x_{0})^{2}/t},
\end{equation}
for some constants $C>0$ and $\kappa\in(0,1]$ only depending on $\delta>0$, where the latter can be taken sufficiently small such that $\int_{0}^{\infty}e^{\delta x^2}\nu_0(dx)<\infty$.
\begin{rem}
	The estimate (\ref{absorbing_density_est}) is derived in \cite[Sec.~6]{hambly_sojmark_2018} by first showing that
	\[
	\widetilde{V}_{t}(x,x_0)=G_{0,t}^{\sigma}(x,x_{0})+\int_{0}^{t}\mathbb{E}_{\widetilde{X}_{0}=x_0} [b(s,\widetilde{X}_{s})\partial_{y}G_{s,t}^{\sigma}(x,\widetilde{X}_{s})\mathbf{1}_{s<\widetilde{\tau}}]ds,
	\]
	where $G_{\cdot,t}^{\sigma}$ is the Green's function for $\partial_{s}u(s,x)=-\frac{1}{2}\sigma(s)^{2}\partial_{xx}u(s,x)$ as a terminal-boundary value problem on $[0,t)\times\mathbb{R}_+$ with $u(\cdot, 0) =0$. This follows from ideas closely related to those in Lemma \ref{Fixed_Lem_StartingPoint} and the estimate then comes from a careful change of measure in the second term to remove the drift of $\widetilde{X}$.
\end{rem}

With a view towards Section \ref{Sec62}, fix any $\bar{\kappa} < \kappa$ and set $\bar{\beta} := \beta\land\bar{\kappa}$, where $\beta$ is the decay exponent of the initial condition. Then the power law decay of $\widetilde{V}_t(x,x_0)$ near the boundary, cf.~(\ref{absorbing_density_est}), implies that we can replicate the arguments from Section \ref{Sect_FixedPoint} with $\widetilde{V}_{t}$ in place of the Dirichlet heat kernel, $G_t$. In this way Proposition~\ref{Fixed_Prop_Stability} remains true only with $\bar{\beta}$ in place of $\beta$.

\subsection{Density bounds for bootstrapping}
\label{Sec62}

Let $\nu_t$ denote the law of $X_t$ killed at the origin. In order to allow bootstrapping of our short-time results, we need to establish analogoues of Lemma \ref{eq:Boot_IC_recovery_I} and Corollary \ref{Boot_Cor_ExtendSolutionOneStep}.

First, referring to (\ref{absorbing_density_est}) as above, we can argue by analogy with Proposition \ref{MinimalJumps_Prop_DensityExists} to see that $\nu_t$ has a density, $V_t$. Moreover, these same arguments show that $V_t$ is bounded with $\left\Vert V_{t}\right\Vert _{\infty}\leq\left\Vert V_{0}\right\Vert _{\infty}+C\int_{0}^{\infty}e^{\delta x^2}\nu_0(dx)$.

Now suppose $L \in H^{1}(0,t_0)$, for some $t_0>0$, and fix $t<t_0$. As in Section \ref{Sect_Bootstrap}, we can then define, for all $s \leq t$, the dual process
\[
d\widehat{X}_{s} = -b(t-s,X_{t-s})ds+\sigma(t-s)dB_{s}+\alpha d\widehat{L}_{s},
\]
where $\widehat{L}_{s}=L_{t-s}$ and $B$ is a Brownian motion. As in the proof of Lemma \ref{eq:Boot_IC_recovery_I},
it follows from Hunt's switching identity that
\[
V_{t}(x)=\mathbb{E}_{\widehat{X}_{0}=x} [V_{0}(\widehat{X}_{t})\mathbf{1}_{t<\widehat{\tau}} ],\quad\widehat{\tau}=\inf\bigl\{ s\geq0:\widehat{X}_{s}\leq0\bigr\}.
\]
Thus, by the same Girsanov argument as in the extension of Lemma \ref{Fixed_Lem_Girsanov} above, it holds for any $q>1$ that
\begin{equation}
\label{Vt_Sec6.2}
V_{t}(x)\leq C_{q}\mathbb{E}_{Y_0=x}\left[V_{0}(Y_{t})^{q}\mathbf{1}_{t<\tau_Y}\right]^{\frac{1}{q}},
\end{equation}
where
\[
dY_{s} = -b(t-s,X_{t-s})ds+\sigma(t-s)dB_{s},
\]
and $\tau_Y$ is the first hitting time of zero by $Y$.  As for $X$, the absorbed process $Y$ has a transition density $U_t(y,y_0)$, so we get
\begin{equation}
\label{Hunt_Sec6.2}
\mathbb{E}_{Y_0=x}\left[V_{0}(Y_{t})^{q}\mathbf{1}_{t<\tau_Y}\right] = \int_{0}^{\infty}V_{0}(y)^{q}U_{t}(y,x)dy.
\end{equation}
Appealing again to \cite{hambly_sojmark_2018}, we see that $U$ satisfies precisely the same bound as (\ref{absorbing_density_est}). Now observe that $(t^{-1}xy\land 1)\leq t^{-\beta q}x^{\beta q}y^{\beta q}$ for $\beta q\leq1$ and recall that $V_0(y)\leq Cy^\beta$. Using this and the order $x^{\kappa}$ boundary decay of the second term in (\ref{absorbing_density_est}), as applied to $ U_t(y,x)$, it follows easily from (\ref{Vt_Sec6.2}) and (\ref{Hunt_Sec6.2}) that
\[
V_{t}(x)\leq C'_q \bigl( x^{\beta} + x^{\kappa/q} \bigr), \quad \text{for all}\;t<t_0,
\]
for any $q>1$ such that  $\beta q\leq1$. Taking $q$ close enough to $1$ so that $\kappa / q \leq \bar{\beta}$ (with $\bar{\beta}$ fixed in Section \ref{Sec61}) we can conclude that $V_t$ satisfies condition (\ref{eq:Intro_DensityAssumption}) for every $t \in [0,t_0)$ with $\bar{\beta}$ as the decay exponent. Given this, we immediately obtain the desired analogues of Lemma \ref{eq:Boot_IC_recovery_I} and Corollary \ref{Boot_Cor_ExtendSolutionOneStep}.

\subsection{Contractivity of the fixed point map}
\label{Sec63}

In this subsection we finally show that $\Gamma$ remains an $L^{\infty}$ contraction, by extending the first part of Theorem 1.6, and then we complete the proof of Theorem \ref{Thm_GenUniqueness}.

With $\ell, \bar{\ell} \in \mathcal{S}(\gamma, A,t_0)$ as in the statement of Theorem \ref{Intro_Thm_Minimality}, consider for $t \in [0,t_0)$ the corresponding processes $X^{\ell}$ and $X^{\bar{\ell}}$ given by
\[
dX_{t}^{\ell}=b(t,X_{t}^{\ell})dt+\sigma(t)dB_{t}-\alpha d\ell_t, \quad dX_{t}^{\bar{\ell}}=b(t,X_{t}^{\bar{\ell}})dt+\sigma(t)dB_{t}-\alpha d\bar{\ell}_t,
\]
coupled through the same Brownian motion and initial condition. Let $\tau_{\ell}$ and $\tau_{\bar{\ell}}$ denote the respective hitting times of zero.

Since $\bar{\ell}$ is increasing, and recalling also the upper bound on $b$ for $x\geq0$, it holds for all $s \leq u < \tau_{\bar{\ell}}$ that
\begin{equation}
\label{eq:comaprison_1}
X_{u}^{\bar{\ell}}-X_{s}^{\bar{\ell}} 
\leq \bar{C}\int_{s}^{u}\sigma(r)^{2}dr+\int_{s}^{u}\sigma(r)dB_{r} =: \bar{X}_{s,u}.
\end{equation}
Moreover, by the Lipschitzness of the drift we have
\begin{align}
|X_{t}^{\bar{\ell}} - X_{t}^{\ell}| \leq C \int_{0}^{t}|X_{r}^{\bar{\ell}}-X_{r}^{\ell}|dr+\alpha \Vert \ell -\bar{\ell} \Vert_{L^\infty(0,t)}, \nonumber
\end{align}
so Gr\"onwall's lemma yields
\[
|X_{t}^{\bar{\ell}} - X_{t}^{\ell}| \leq \alpha e^{Ct} \Vert \ell -\bar{\ell} \Vert_{L^\infty(0,t)}.
\]
Therefore, on the event $\left\{ \tau_{\ell}=s \right\}$ it holds for all $s<t_0$ that
\begin{equation}
\label{eq:comaprison_2}
X_{s}^{\bar{\ell}}=X_{s}^{\bar{\ell}}-X_{s}^{\ell}\leq \alpha C_{0}\Vert \ell -\bar{\ell} \Vert_{L^\infty(0,s)}.
\end{equation}

Using (\ref{eq:comaprison_1}) and (\ref{eq:comaprison_2}), and arguing as in (\ref{eq:Unique_MainMinimalArgument}), we deduce that
\begin{align}
\Gamma[\ell]_t - \Gamma[\bar{\ell}]_t 
&\leq \int_{0}^{t}\mathbb{P}\bigl(\inf_{u\in[s,t]}\bar{X}_{s,u}+ \alpha C_{0} \Vert \ell -\bar{\ell} \Vert_{L^\infty(0,s)} >0\bigr)d\Gamma[\ell]_s \nonumber \\
&= \int_{0}^{t} \int_{0}^{\infty}\bar{p}_{t-s}\bigl(x,\alpha C_{0}\Vert \ell -\bar{\ell} \Vert_{L^\infty(0,s)} \bigr)dx d\Gamma[\ell]_s,
\nonumber
\end{align}
for $t < t_0$, where a simple time-change shows that $\bar{p}_{r}(x,x_0)$ is given explicitly by
\[
\bar{p}_{r}(x,x_0)=\frac{1}{\sqrt{2\pi \varsigma(r)}}\Bigl(1-\exp\Bigl\{-\frac{2xx_{0}}{\varsigma(r)}\Bigr\}\Bigr)\exp\Bigl\{-\frac{(x-x_{0}-\bar{C}\varsigma(r))^{2}}{2\varsigma(r)}\Bigr\}
\]
with $\varsigma(r)=\int_{s}^{r+s}\sigma(u)^{2}du$. Noting that $c (t-s) \leq \varsigma(t-s) \leq C (t-s)$ and using the bound $1-e^{-z} \leq z\wedge1$, it is straightforward to verify that
\[
\int_{0}^{\infty}\bar{p}_{t-s}(x,\alpha C_{0}y)dx \leq C'_0 (t-s)^{-\frac{1}{2}} y, \qquad \text{for} \;\; y\geq0.
\]
Thus, it holds for all $t < t_0$ that
\[
\Gamma[\ell]_t - \Gamma[\bar{\ell}]_t 
\leq C'_0 \Vert \ell -\bar{\ell} \Vert_{L^\infty(0,t)} \int_{0}^{t}(t-s)^{-\frac{1}{2}}d\Gamma[\ell]_{s}. 
\]

From here we can repeat the proof of Theorem \ref{Intro_Thm_Minimality} to get contractivity of $\Gamma$ and by Section \ref{Sec61} we have stability. Thus, the short-time version of Theorem \ref{Thm_GenUniqueness} holds by the same arguments as in the proof of Theorem \ref{Intro_Thm_FixedPoint}. Finally, Section \ref{Sec62} allows us to bootstrap this result up to the $H^1$ explosion time and so the full Theorem \ref{Thm_GenUniqueness} follows.

\subsection{Minimality, monotonicity and trapping}
\label{Sec64}
In this final subsection we consider the special case $b(t,x)=b(t)$ and sketch the remaining steps towards the proof of Theorem \ref{Thm_Generic_uniqueness_b(t)}. Proceeding as in Section \ref{Sec63}, the Gr\"onwall argument is now no longer needed and instead we immediately get $X_{s}^{\bar{L}}=X_{s}^{\bar{L}}-X_{s}^{L}\leq \alpha (L_s -\bar{L}_s)^+$ on $\left\{ \tau_{L}=s \right\}$. Therefore, the ensuing arguments yield
\[
(L_t - \bar{L}_t)^+ \leq  C' \Vert (L -\bar{L})^+ \Vert_{L^\infty(0,t)}  \int^t_0 (t-s)^{-\frac{1}{2}} dL_s
\]
and hence the exact same reasoning as in the second part of Theorem \ref{Intro_Thm_Minimality} ensures that any differentiable solution $L \in \mathcal{S}(\gamma,A,t_0)$ is minimal. In terms of the monotonicity and trapping procedure, the only change in Lemma \ref{Boot_Lem_Monotonicity} is that equation (\ref{eq:Boot_ComparisonII}) now reads as
\begin{equation}
L^\varepsilon_t \geq L_t + \mathbb{P}(\inf_{0\leq u \leq t} \{ X_0 + {\textstyle \int_0^u} b(r) dr + {\textstyle \int_0^u} \sigma(r) dB_r - f_u \} \in (\tfrac{1}{16}\varepsilon, \tfrac{3}{16} \varepsilon)).
\end{equation}
Consequently, the proof of Lemma \ref{Boot_Lem_Monotonicity}  goes through by absolute continuity with respect to the time-changed Brownian motion $u\mapsto\int_0^u \sigma(r) dB_r$. Furthermore, the proof of Lemma \ref{Boot_Lem_Convergence} follows immediately by applying the same reasoning as in Section \ref{Sec63}. Given this, the proof of Theorem \ref{Thm_Generic_uniqueness_b(t)} can be finished in precisely the same way as the second part of Theorem \ref{Intro_Thm_MainUniqueness}.


\paragraph*{Acknowledgements.}
The research of A. S{\o}jmark was partially supported by the EPSRC award EP/L015811/1. We are grateful for the careful reading and suggested improvements from an anonymous referee.


\bibliographystyle{alpha}

\end{document}